\documentclass[11pt, a4paper]{article}
\usepackage{amsmath, amssymb, amsthm, paralist, url, mathtools}
\usepackage[all]{xy}
\mathtoolsset{centercolon}

\usepackage{tikz}

\newcommand{\diagnode}[1]{\fill #1 circle (.1);}
\newcommand{\distorbit}[1]{\draw #1 circle (.2);}

\newtheorem{theorem}{Theorem}[section]
\newtheorem{lemma}[theorem]{Lemma}

{\theoremstyle{definition}
\newtheorem{remark}[theorem]{Remark}

\newtheorem{example}[theorem]{Example}
}

\newcommand{\B}{\mathsf{B}}
\newcommand{\F}{\mathsf{F}}

\newcommand{\C}{\mathbb{C}} 

\newcommand{\G}{\mathbf{G}} 

\newcommand{\PG}{\mathsf{PG}}
\newcommand{\PSL}{\mathsf{PSL}}
\newcommand{\PGO}{\mathsf{PGO}}

\newcommand{\Z}{\mathbb{Z}}
\newcommand{\oct}{\mathcal{O}}

\newcommand{\mouf}{\mathbb{M}}
\newcommand{\lie}{\mathcal{L}}
\newcommand{\cartan}{\mathcal{H}}
\newcommand{\I}{{I}}
\newcommand{\mixed}{\mathrm{mixed}}

\DeclareMathOperator{\ch}{char}
\DeclareMathOperator{\tr}{T}
\DeclareMathOperator{\R}{\mathbb{R}}
\DeclareMathOperator{\N}{N}

\DeclareMathOperator{\Fix}{Fix}
\DeclareMathOperator{\Stab}{Stab}

\DeclareMathOperator{\diag}{diag}
\DeclareMathOperator{\ex}{exp}

\DeclareMathOperator{\Sym}{Sym}

\numberwithin{equation}{section}

\begin{document}

\title{Moufang sets of mixed type $\F_4$}
%
%

\author{Elizabeth Callens \and Tom De Medts}

\date{\today}

\maketitle

\begin{abstract}
	Moufang sets were introduced by Jacques Tits in order to understand isotropic linear algebraic groups of relative rank one,
	but the notion is more general.
	We describe a new class of Moufang sets, arising from so-called mixed groups of type $\F_4$ in characteristic $2$,
	obtained as the fixed point set under a suitable involution.
\end{abstract}

\section{Introduction}

Moufang sets were introduced by Jacques Tits in \cite{Durham} as an axiomization of the isotropic simple algebraic
groups of relative rank one, and they are, in fact, the buildings corresponding to these algebraic groups,
together with some of the group structure (which comes from the root groups of the algebraic group).
In this way, the Moufang sets are a powerful tool to study these algebraic groups.
On the other hand, the notion of a Moufang set is more general, and includes many more examples that do not directly
arise by this procedure. In fact, it is still a wide open question whether every Moufang set is, in some sense,
of algebraic origin.

In this paper, we are studying Moufang sets arising from so-called {\em mixed groups} of type $\F_4$.
These groups exist only over fields of characteristic $2$, and they are defined over a {\em pair}
of fields $(k, \ell)$ such that $\ell^2 \leq k \leq \ell$.
There has been an increasing interest in a systematic study of these inseparable situations over non-perfect fields,
most notably by the recent work on {\em pseudo-reductive groups}
by B.~Conrad, O.~Gabber and G.~Prasad~\cite{CGP}.

Formally speaking, a {\em Moufang set} is a set $X$ together with a collection of groups
$\bigl( U_x \leq \Sym(X) \bigr)_{x \in X}$, such that each $U_x$ acts regularly on $X \setminus \{x\}$,
and such that%
\footnote{We will always write our permutation actions on the right.}
$U_x^\varphi = U_{x.\varphi}$ for all $\varphi \in G^\dagger := \langle U_x \mid x \in X \rangle$.
The groups $U_x$ are called the {\em root groups} of the Moufang set, and the group $G^\dagger$ is called
its {\em little projective group}.

The groups arising from algebraic groups of type $\F_4$ have already been described explicitly in \cite{DVM}.
The techniques used in that paper, however, rely heavily on the fact that the algebraic groups of type $\F_4$ arise
as the automorphism groups of certain $27$-dimensional algebraic structures known as Albert algebras, and such a
description is not available for the mixed groups of type $\F_4$.

We therefore take a completely different approach, inspired by \cite{MVM} and \cite{Steinbach1}, using the description of the corresponding Chevalley groups,
and replacing the geometric ingredients of the algebraic approach (namely polarities of the octonion plane) by
group theoretic ingredients (namely involutions of the Chevalley groups).
More specifically, we show how to construct a split saturated BN-pair of rank one from a well-chosen involution.
Such a BN-pair is essentially equivalent to a Moufang set.

The paper is organized as follows.
In section~\ref{se:setup}, we recall the necessary basics about Chevalley groups, Moufang sets, and BN-pairs.
Section~\ref{se:mixed} deals with the basic theory of mixed groups as introduced by J.~Tits.
In section~\ref{se:mixedgroups}, we specifically look at mixed Chevalley groups of type $\F_4$,
and in section~\ref{se:BN}, we study involutions of these mixed groups such that the centralizer of the involution
is a split BN-pair of rank one.
These BN-pairs give rise to the Moufang sets we are interested in, and in section~\ref{se:mixedF4}
we proceed to explicitly describe these Moufang sets.
This culminates in our main result (Theorem~\ref{th:main}).
In the last section~\ref{se:alg}, we point out that in the algebraic case,
we recover the known description of algebraic Moufang sets of type $\F_4$ as in \cite{DVM}.

\paragraph*{Acknowledgments}

It is our pleasure to thank Bernhard M\"uhlherr, who prophetically predicted the existence of the mixed Moufang sets of type $F_4$,
and shared its mysteries with us.
We are grateful to Hendrik Van Maldeghem for the lively geometric discussions about this topic.
Finally, we thank the referee, for very carefully reading our paper, correcting some of our mistakes and making good suggestions
to fill in certain details that were initially missing.
In particular, the referee pointed out how to improve both the result and the proof of Theorem~\ref{T(K,L)} below;
see also Remark~\ref{rem:referee}.


\section{Setup}\label{se:setup}

\subsection{Chevalley groups}

We briefly recall some basics about Chevalley groups that we will need in the sequel.
Our main reference is \cite{Carter}.

\subsubsection{The Cartan decomposition of a complex simple Lie algebra}

Let $\lie$ be a Lie algebra over $\C$, with Lie bracket $[ \cdot , \cdot ]$.
A Cartan subalgebra $\cartan$ is a subalgebra of $\lie$ which is nilpotent, and such that
$\cartan$ is not contained as an ideal in any larger subalgebra of $\lie$, i.e.\@ 
if $x \in \lie$ is such that $[x,h] \in \cartan$ for all $h \in \cartan$, then $x \in \cartan$.

Now if $\lie$ is simple over $\C$, then $\lie$ can be decomposed into a direct sum of $\cartan$ with a number of
one-dimensional $\cartan$-invariant subspaces:
\[ \lie = \cartan \oplus \lie_{r_1} \oplus \dots \oplus \lie_{r_m} . \]
The one-dimensional subspaces $\lie_{r_i}$ are called the root spaces of $\lie$ (w.r.t.\@~$\cartan$).

In each one-dimensional subspace $\lie_r$, we choose a non-zero element $e_r$. Then for each $h\in\cartan$, we have
\begin{equation}\label{H}
[he_r]=r(h)e_r
\end{equation} 
for some $r(h)\in \C$.
This defines a linear map
\[ r \colon \mathcal{H}\to \C;h\mapsto r(h) . \]
It can be shown that each element $f$ of the dual space $\cartan^*$ of $\mathcal{H}$ is of the form
\[f:\mathcal{H}\to \C;h\mapsto (x,h)\]
for some unique $x\in\mathcal{H}$, where $(\cdot,\cdot)$ is the Killing form on $\mathcal{H}$.
In this way $r$ corresponds to a unique element in $\mathcal{H}$, which we also denote by $r$.
We can repeat this procedure for each root space and denote by $\Psi$ the subset of $\mathcal{H}$ we obtain in this way. 

One can show that $\Psi$ forms a root system in $\mathcal{H}_{\R}$, where $\mathcal{H}_{\R}$ is the set of linear combinations of elements of $\Psi$ with real coefficients.
As the Killing form induces an isomorphism between $\mathcal{H}$ and its dual space $\mathcal{H}^*$,
there is a corresponding root system $\Phi$ in $\mathcal{H}^*$. The elements of $\cartan^*$ are called the {\em roots} of $\lie$ (w.r.t.\@~$\cartan$).

We can rewrite the Cartan decomposition of $\lie$ as
\[ \lie = \cartan \oplus \bigoplus_{r \in \Phi} \lie_r , \]
in such a way that for any pair of roots $r,s \in \Phi$, we have
\begin{alignat*}{2}
	[\lie_r, \lie_s] &= \lie_{r+s} \quad  && \text{ if } r+s \in \Phi , \\
	[\lie_r, \lie_s] &= 0 && \text{ if } r+s \not\in \Phi, r+s \neq 0 , \\
	[\lie_r, \lie_{-r}] &= \C r , \\
	[\cartan, \lie_r] &= \lie_r .
\end{alignat*}
These relations can be made more precise.
Indeed, if $r$ is any root, then the element $h_r\in\cartan$ corresponding to $(2r)/\langle r,r\rangle$ under the isomorphism is called the {\em coroot} of $r$.
Now let $\Pi$ be a set of fundamental roots for $\Phi$;
if we choose $e_{-r}\in\lie_{-r}$ such that $[e_r,e_{-r}]=h_r$ for each $r \in \Phi$ then
\begin{equation}\label{eq:Chevalleybasis}
	\{ h_r \mid r \in \Pi \} \cup \{ e_r \mid r \in \Phi \}
\end{equation}
forms a basis for $\lie$, called a {\em Chevalley basis}, satisfying
\begin{alignat*}{2}
	[h_r, h_s] &= 0 , \\
	[h_r, e_s] &= A_{rs} e_s , \\
	[e_r, e_{-r}] &= h_r , \\
	[e_r, e_s] &= 0 && \text{ if } r+s \not\in \Phi, r+s \neq 0 , \\
	[e_r, e_s] &= N_{rs} e_{r+s} \quad  && \text{ if } r+s \in \Phi .
\end{alignat*}
The constants $A_{rs}$ are easily determined by the root system,
as is the absolute value of the constants $N_{rs}$;
determining the {\em sign} of the $N_{rs}$ is much more delicate, however.
Since we will be working over fields of characteristic~$2$, we need not worry about these signs.

\subsubsection{Chevalley groups}

Let $\lie$ be a simple Lie algebra over $\C$ with Chevalley basis as in~\eqref{eq:Chevalleybasis}.
Now let $\lie_{\Z}$ be the subset of $\lie$ of all {\em integral} linear combinations of the basis elements;
then $\lie_{\Z}$ becomes a Lie algebra over $\Z$.

Now let $k$ be any field.
Then we can form the tensor product of the additive group of $k$ with the additive group of $\lie_\Z$, and define
\[ \lie_k = k \otimes \lie_{\Z} , \]
which is in a natural way a Lie algebra over $k$.

We now introduce certain automorphisms of $\lie_k$.
For every root $r \in \Phi$ and every element $t \in k$, we define an automorphism $u_r(t)$ as follows:
\begin{alignat*}{3}
	&u_r(t) \cdot e_r &&= e_r, \\
	&u_r(t) \cdot e_{-r} &&= e_{-r} + t h_r - t^2 e_r, \\
	&u_r(t) \cdot e_s &&= \sum_{i=0}^q M_{r,s,i} t^i e_{ir+s} \quad && \text{ if } r,s \text{ are linearly independent}, \\
	&u_r(t) \cdot h_s &&= h_s - A_{sr} t e_r \quad && \text{ for } s \in \Pi ,
\end{alignat*}
where in the rule for $u_r(t) \cdot e_s$, $q$ is the largest integer such that $qr + s \in \Phi$, and where the constants
$M_{r,s,i}$ are defined in terms of the structure constants $N_{rz}$.

The {\em (adjoint) Chevalley group} of type $\lie$ over $k$ is now defined as the group
\[ \lie(k) := \langle u_r(t) \mid r \in \Phi, t \in k \rangle . \]
It turns out that this group is independent of the choice of the Chevalley basis;
its isomorphism type depends only on $\lie$ and $k$.
Note that $u_r(s) u_r(t) = u_r(s+t)$ for all $r \in \Phi$ and all $s,t \in k$.
The subgroups $U_r = \{ u_r(t) \mid t \in k \}$ are called the {\em root subgroups} of $G = \lie(k)$.

\begin{remark}
\begin{compactenum}[(i)]
    \item
	If $k = \C$, then $u_r(t) = \exp(t \operatorname{ad} e_r)$, and in fact, this is where the definition of the automorphisms
	$u_r(t)$ in the general case comes from.
    \item
	We have been following Chevalley's original approach to construct Chevalley groups.
	This construction has later been generalized to include not only adjoint groups but also more general
	connected semisimple split linear algebraic groups.
	The corresponding building, and consequently also the Moufang set that we will construct, does not detect this distinction
	(its little projective group always corresponds to the adjoint representation)
	so it is no loss of generality to restrict to the adjoint case.
    \item
	The relation between (not necessarily adjoint) Chevalley groups and linear algebraic groups is as follows.
	Let $k$ be an arbitrary field, and let $K$ be its algebraic closure.
	Then a Chevalley group $\lie(K)$ is a connected semisimple linear algebraic group $\G$ over~$K$ of type $\lie$,
	defined and split over $k$ (and in fact over the prime subfield of $k$).
	The Chevalley group $\lie(k)$ is the commutator subgroup of the group $\G(k)$ of $k$-rational points of $\G$.
\end{compactenum}
\end{remark}

An important feature of Chevalley groups and of linear algebraic groups is that the root groups satisfy
certain {\em commutator relations}.
We will discuss those relations (and the extension of this concept to mixed groups) in section~\ref{se:mixed} below.

\subsubsection{Weyl groups and subgroups of Chevalley groups}\label{Chevalley}

We introduce some notation for certain important subgroups of Chevalley groups that we will need in the future.

Let $\Phi$ be the root system associated to an arbitrary Chevalley group $\lie(k)$, let $\Pi$ be a fundamental root system of $\Phi$, $\Phi^+$ be the set of positive roots and $\Phi^{-}$ be the set of negative roots of $\Phi$.
One can associate to every root $r\in\Phi$ a reflection $w_r$; the group generated by all these reflections is called the Weyl group $W$ of $\lie(k)$.
More information on Weyl groups can be found in \cite[Chapter 2]{Carter}.

Furthermore, every Chevalley group $G$ has an associated {\em BN-pair}, i.e.\@ 
a pair of subgroups $(B,N)$ of $G$ such that the following conditions hold.
\begin{compactenum}[(1)]
\item $G=\langle B,N\rangle$.
\item $B\cap N$ is a normal subgroup of $N$.
\item $W=N/B\cap N$ is generated by elements $w_i$ such that $w_i^2=1$, $i\in I$.
For each $i \in I$, we choose a preimage $n_i \in N$ of $w_i$.
\item For each $i \in I$ and each $n \in N$, we have
\[
B n_i B \cdot B n B \subseteq B n_i n B \cup B n b.
\]
\item For each $i \in I$, we have $n_i B n_i \neq B$.
\end{compactenum}
For each $J \subseteq I$, we define
\[ W_J := \langle w_j \mid j \in J \rangle . \]
A subgroup $P$ of $G$ is called a parabolic subgroup if $P$ contains $B$ or some conjugate of $B$.
One can show that the parabolic subgroups containing $B$ are of the form $B N_J B$ with $N_J$ the preimage of $W_J$ under the natural projection from $N$ to $W$.
A consequence of the axioms is that the group $G$ has a so-called {\em Bruhat decomposition} $G=BNB$.

In our setting of Chevalley groups,
we define $n_r(t):=u_r(t)u_{-r}(-t^{-1})u_{r}(t)$, $n_r:=n_r(1)$ and $h_r(t):=n_r(t)n_r(-1)$ for all $r\in\Phi$ and all $t\in k^\times$.
Let
\begin{align*}
N&:=\langle n_r(t)\mid r\in\Phi, t\in k^\times \rangle,\\
H&:= \langle h_r(t)\mid r\in\Phi, t\in k^\times \rangle,\\
U&:=\langle u_r(t)\mid r\in\Phi^{+}, t\in k\rangle\ \text{and}\\
B&:=UH .
\end{align*}
Then one can show that $(B,N)$ forms a BN-pair for $\lie(k)$ with $B\cap N=H$ and $N/H=W$ with $W$ the Weyl group of $\lie(k)$. 

Finally, we introduce some standard notation. Let $J$ be a subset of $\Pi$, then we define $\Phi_J$ as $\Phi\cap \langle J\rangle$ and $W_J$ as the Weyl group generated by all $w_{\alpha}$ with $\alpha\in J$. We denote by $w_0$ the longest element in $W$ and similarly $w_0^{J}$ is the longest element in $W_J$.
We then define
\begin{align*}
	U_J &:= \langle U_r\mid r\in\Phi^{+}\setminus \Phi_J\rangle, \\
	L_J &:= \langle H,U_r\mid r\in \Phi_J\rangle, \\
	P_J &:= U_JL_J.
\end{align*}

\subsection{Moufang sets}

In this section, we recall some of the basics of Moufang sets, and we refer to \cite{course} for more details.

A {\em Moufang set} $\mouf = \bigl( X, (U_x)_{x \in X} \bigr)$ is a set $X$ together with a collection of groups
$U_x \leq \Sym(X)$, such that for each $x \in X$:
\begin{compactenum}[(1)]
    \item
	$U_x$ fixes $x$ and acts sharply transitively on $X \setminus \{ x \}$;
    \item
	$U_x^\varphi = U_{x\varphi}$ for all $\varphi \in G := \langle U_z \mid z \in X \rangle$.
\end{compactenum}
The group $G$ is called the {\em little projective group} of the Moufang set.

A typical example is given by the group $G = \PSL(2,k)$ acting on the projective line $X = \mathbb{P}^1(k) = k \cup \{ \infty \}$.

\subsubsection{Moufang sets and algebraic groups of relative rank one}

One of the main motivations (but certainly not the only one) for studying Moufang sets, is that they provide a tool to understand
linear algebraic groups of relative rank one.
We will briefly explain the connection.

So suppose that $\G$ is an absolutely simple algebraic group defined over a field $k$, and assume that $\G$ has $k$-rank one.
Let $X$ be the set of all $k$\nobreakdash-parabolic subgroups of $\G$.
For each $x \in X$, we let $U_x$ be the root subgroup of the $k$-parabolic subgroup $x$
(which coincides with the $k$-unipotent radical of $x$).
Then $\bigl( X, (U_x)_{x \in X} \bigr)$ is a Moufang set, which we will denote by $\mouf(\G, k)$.
If we define $\G^+(k)$ to be the subgroup of $\G(k)$ generated by all the root subgroups,
then $\G^+(k)$ modulo its center acts faithfully on $X$.

As an example, we could consider groups of the form $\G = \PSL(2,D)$, where $D$ is a division algebra of degree $d$ over a field $k$;
in this case, $\G$ is an algebraic group of type $\mathsf{A}_{2d-1}$ of $k$-rank one.
Note that $\PSL(2,D)$ still gives rise to a Moufang set if $D$ is infinite-dimensional over its center,
but in this case the Moufang set no longer arises from an algebraic group.

\subsubsection{An explicit construction of Moufang sets}

We will now explain how any Moufang set can be reconstructed from a single root group together with one additional permutation \cite{MoufJor}.

Let $(U, +)$ be a group, with identity $0$, and where the operation $+$ is {\em not necessarily commutative}.
Let $X = U \cup \{ \infty \}$, where $\infty$ is a new symbol.
For each $a\in U$, we define a map $\alpha_a \in \Sym(X)$ by setting
\begin{equation}\label{eq:alpha}
	\alpha_a \colon \begin{cases}
		\infty \mapsto \infty \\
		x \mapsto x+a & \text{ for all $a \in U$}.
	\end{cases}
\end{equation}
Let
\begin{equation*}
U_\infty := \{\alpha_a\mid a\in U\} \,.
\end{equation*}
Now let $\tau$ be a permutation of $U^*$.
We extend $\tau$ to an element of $\Sym(X)$ (which we also denote by $\tau$) by setting $0^\tau=\infty$ and $\infty^\tau=0$.
Next we set
\begin{equation}\label{eq:defU}
	U_0 := U_\infty^\tau \text{ and } U_a := U_0^{\alpha_a}
\end{equation}
for all $a\in U$, where $U_x^\varphi$ denotes conjugation inside $\Sym(X)$. Let
\begin{equation}\label{eq:mouf}
	\mouf(U,\tau) := \bigl( X,(U_x)_{x\in X} \bigr)
\end{equation} 
and let
\begin{equation*}
	G := \langle U_\infty,U_0 \rangle = \langle U_x \mid x \in X \rangle \,.
\end{equation*}
Then $\mouf(U, \tau)$ is not always a Moufang set, but every Moufang set can be obtained in this way.
The next lemma shows us how to do this.
\begin{lemma}\label{Utau}
Let $\mouf=(X,(U_x)_{x\in X})$ be Moufang set. Pick two elements $0,\infty\in X$ and define $U$ as $X\setminus\{\infty\}$.\\
For every $a\in U$, define $\alpha_a\in U_{\infty}$ as the unique element such that $\alpha_a(0)=a$. Let $a+b:=\alpha_b(a)$ for every $a,b\in U$ and $\tau\in\Sym(X)$ be a permutation interchanging $0$ and $\infty$ such that $U_{\infty}^{\tau}=U_{0}$. Then $\mouf=\mouf(U,\tau)$.
\end{lemma}
\begin{proof}
This is obvious from the above construction of $\mouf(U, \tau)$.
\end{proof}
Note that, for a given Moufang set, the map $\tau$ is certainly not unique: different choices for $\tau$ can give rise to the same Moufang set.

\subsubsection{Split BN-pairs of rank one}

We introduce the notion of a saturated split BN-pair of rank one because of the correspondence with Moufang sets. In the context of Chevalley groups, it sometimes is more natural to work with BN-pairs. We show how to construct a Moufang set from these BN-pairs. 

A {\em BN-pair of rank one} in a group $G$ is a system $(B,N)$ of two subgroups $B$ and $N$ of $G$ such that
the following axioms hold:
\begin{compactenum}[(i)]
    \item
	$G=\langle B,N\rangle$.
    \item
	$H:=B\cap N \unlhd N$.
    \item
        There is an element $\omega\in N\setminus H$ with $\omega^2\in H$ such that $N=\langle H,w\rangle$, $G=B\cup B\omega B$ and $\omega B\omega\neq B$.
\end{compactenum}
We call such a pair split if additionally
\begin{compactenum}[(i)]\setcounter{enumi}{3}
    \item
       There exists a normal subgroup $U$ of $B$ such that $B=U\rtimes H$.
\end{compactenum}
holds, and saturated if additionally
\begin{compactenum}[(i)]\setcounter{enumi}{4}
    \item
       $H=B\cap B^{\omega}$.
\end{compactenum}

\begin{lemma}\label{BN-pair}
Let $G$ be a group with a saturated split BN-pair of rank one, let $X:=\{U^g\mid g\in G\}$ be the set of conjugates of $U$ in G. Denote by $V_x$, the element $x\in X$ viewed as a subgroup of $G$. Then $(X,(V_x)_{x\in X})$ is a Moufang set.
\end{lemma}
\begin{proof}
For a proof, see \cite[Proposition 2.1.3.]{course}.
\end{proof}

Using the alternative definition of a Moufang set, we find that a representation of the Moufang set corresponding to this BN-pair is $\mouf(U,\omega)$.

\subsection{Algebraic Moufang sets of type \boldmath $\F_4$} \label{f4}

We will now give an easy description of the Moufang sets arising from algebraic groups of type $\F_4$ using the method we have explained
in the previous section.

\begin{theorem}\label{thm:algF4}
	Let $k$ be an arbitrary field.
	Every Moufang set of type $\F_4$ over $k$ is determined by an octonion division algebra $\oct/k$.
	More precisely, if $\oct$ is such an octonion division algebra, then we define
	\[ U := \{ (a,b) \in \oct \times \oct \mid \N(a) + \tr(b) = 0 \} , \]
	where $\N$ and $\tr$ are the standard norm and trace maps from $\oct$ to $k$.
	We define the (non-abelian) group operation $+$ on $U$ by setting
	\[ (a,b) + (c,d) := (a+c, b+d - \overline{c}a) \]
	for all $(a,b), (c,d) \in U$.
	Finally, we define a permutation $\tau \in \Sym(U^*)$ by setting
	\[ (a,b)^\tau := \bigl( -ab^{-1}, b^{-1} \bigr) \]
	for all $(a,b) \in U$.
	Then the corresponding Moufang set of type $\F_4$ is equal to $\mouf(U, \tau)$.
\end{theorem}
\begin{proof}
	See \cite{DVM}.
\end{proof}

The goal of this paper is to extend this result to so-called mixed groups of type $\F_4$ (which we will define in the next section).
We will see, however, that we will not only have to use very different methods, but that also the resulting description will
not simply be of the same form as the nice description that we have in Theorem~\ref{thm:algF4};
see Theorem~\ref{th:main} below.

\section{Mixed groups}\label{se:mixed}

In this section, we recall some basic facts about mixed groups.
Our main reference is \cite[Section 10.3]{Tits74}.

Let $G$ be an adjoint split simple algebraic group of type $X$ defined over a field $k$ of characteristic $p$,
where either $X = \B_n, \mathsf{C}_n, \F_4$ and $p=2$, or $X = \mathsf{G}_2$ and $p=3$.
Assume moreover that $\ell$ is a field such that $\ell^p \leq k \leq \ell$.

Let $T$ be a maximal $k$-split torus, let $N = N_G(T)$ be the normalizer of $T$ in~$G$,
let $B$ be a Borel subgroup of $G$ containing $T$,
and let $\Phi$ be a root system of type $X$ corresponding to the maximal torus $T$.
Since $X$ is not simply laced, $\Phi$ consists of long and short roots, and we write
$\Phi = \Phi_\ell \cup \Phi_s$, where $\Phi_\ell$ and $\Phi_s$ denote the sets of long and short roots, respectively.
For each root $r \in \Phi$, we have a corresponding root group $U_r$, i.e.\@ a one-dimensional $k$-unipotent subgroup
of $G$ acted upon by $T$.
In the algebraic group $G$, all root groups are isomorphic to the additive group $\mathbb{G}_a$.
For each $r \in \Phi$, we choose an isomorphism $u_r$ from $\mathbb{G}_a$ to $U_r$.
We also define $\Phi^+$ to be the set of positive roots of $\Phi$, i.e.\@ the roots $r \in \Phi$ such that $U_r \subseteq B$;
correspondingly, we write $\Phi_\ell^+ := \Phi_\ell \cap \Phi^+$ and $\Phi_s^+ := \Phi_s \cap \Phi^+$.

Now let \\[-4ex]
\begin{align*}
	T(k,\ell) &:= \biggl\{ t \in T \;\biggl\lvert\,
	\begin{aligned}
		&r(t) \in k &&\text{for all } r \in \Phi_l \ \text{ and } \\
		&r(t) \in \ell &&\text{for all } r \in \Phi_s
	\end{aligned}
	\biggr\} \,, \\[.3ex]
	N(k,\ell) &:= N(k) \; T(k,\ell) \,, \\
	B(k,\ell) &:= \langle T(k,\ell) \cup \{ U_r(k) \mid r \in \Phi_l^+ \} \cup \{ U_r(\ell) \mid r \in \Phi_s^+ \} \rangle \,,
\intertext{and finally}
	G(k,\ell) &:= \langle T(k,\ell) \cup \{ U_r(k) \mid r \in \Phi_l \} \cup \{ U_r(\ell) \mid r \in \Phi_s \} \rangle \,.
\end{align*}
The group $G(k,\ell)$ is the {\em mixed group} of type $X$ corresponding to the pair of fields $(k,\ell)$, and it is also denoted
by $X(k,\ell)$, particularly when $X$ is specified. One can show that the pair $(B(k,\ell), N(k,\ell))$ forms a BN-pair of $G(k,\ell)$.

\begin{example}[{\cite[p.\@ 204]{Tits74}}]\label{ex:mixedq}
	Let $(k,\ell)$ be a pair of fields of characteristic $2$ with $\ell^2 \leq k \leq \ell$, and
	let $q$ be the ``mixed quadratic form'' from $k^{2n} \times \ell$ to $k$ given by
	\[ q(x_0, x_1, \dots, x_{2n-2}, x_{2n-1}, x_{2n}) = x_0 x_1 + \dots + x_{2n-2} x_{2n-1} + x_{2n}^2 \,. \]
	Then the mixed group $B_n(k,\ell)$ is isomorphic to the group $\mathsf{PGO}(q)$, i.e.\@ the quotient of the group
	of all invertible similitudes of $q$ by the subgroup $k^\times$.
	This group is also isomorphic to the mixed group $\mathsf{C}_n(\ell^2, k)$.

	When we are considering the corresponding building, i.e.\@ the ``mixed quadric'' consisting of the isotropic vectors of $q$,
	it will often be convenient to drop the last coordinate $x_{2n}$, since it is uniquely determined
	from the other coordinates by the equation $q(x_0,\dots,x_{2n-1},x_{2n}) = 0$.
	Thus, the mixed quadric will then consist of points in $\PG(2n-1,k)$ with (projective) coordinates
	$(X_0,\dots,X_{2n-1})$ satisfying the condition
	\begin{equation}
		X_0 X_1 + \dots + X_{2n-2} X_{2n-1} \in \ell^2 \,,
	\end{equation}
	and the higher-dimensional objects of the building are now simply
	the subspaces of the underlying projective space lying on this mixed quadric.
\end{example}

For algebraic groups, it is well known that the root groups satisfy certain commutator relations depending on the root system.
More precisely, it is possible to renormalize the parametrizations $u_r$ in such a way that there are
constants $c_{\lambda,r,\mu,s} \in \{ \pm 1, \pm 2, \pm 3 \}$, called the {\em structure constants}, such that
\begin{equation}\label{eq:comm}
	[u_r(x), u_s(y)] = \prod_{\substack{\lambda,\mu \in \Z_{>0} \\ \lambda r + \mu s \in \Phi}} u_{\lambda r + \mu s}\left(c_{\lambda,r,\mu,s} \cdot x^\lambda y^\mu \right)
\end{equation}
for all $r,s \in \Phi$ and all $x,y \in k$;
see, for example, \cite[Propositions 9.2.5 and 9.5.3]{Springer}, or \cite[Theorem 5.2.2]{Carter} for the analogous statement for Chevalley groups.

This goes through for mixed groups without any change:
we get the same commutator relations \eqref{eq:comm}, but this time
for all $r,s \in \Phi$ and all $x,y \in k$ or $\ell$ depending on whether the corresponding roots $r$ and $s$
are long or short, respectively.
Observe that the condition $\ell^p \leq k \leq \ell$
is exactly the condition which is needed for these commutator relations to make sense,
i.e.\@ the elements $x^\lambda y^\mu$ belong to $k$ whenever the root $\lambda r + \mu s$ is a long root.

In the case $p = \ch(k) = 2$, which will be the only case we will be dealing with in this paper,
the constants $c_{\lambda, r, \mu, s}$ are all equal to $0$ or $1$, so equation \eqref{eq:comm} simplifies further.
In the case $p = 2$ and $X = \F_4$, which is the case that we are interested in in this paper,
we can summarize the commutator relations as follows; see, for instance, \cite[(2.2)--(2.5)]{Ree}:
\begin{align}\label{eq:comm2}
\begin{split}
&\begin{alignedat}{2}
	[u_r(x), u_s(y)] &= 1
		&& \text{ if } r,s \in \Phi \text{ but } r+s \not\in \Phi \,, \\
	[u_r(x), u_s(y)] &= u_{r + s}(xy)
		&& \text{ if } r,s \in \Phi_s \text{ and } r+s \in \Phi_s \,, \\
	[u_r(x), u_s(y)] &= 1
		&& \text{ if } r,s \in \Phi_s \text{ and } r+s \in \Phi_\ell \,, \\
	[u_r(x), u_s(y)] &= u_{r + s}(xy)
		&& \text{ if } r,s \in \Phi_\ell \text{ and } r+s \in \Phi_\ell \,, \\
	[u_r(x), u_s(y)] &= u_{r + s}(xy) \, u_{2r + s}(x^2 y)
\end{alignedat} \\
&\hspace*{18ex} \text{ if } r \in \Phi_s, s \in \Phi_\ell \text{ and } r+s \in \Phi_s, 2r + s \in \Phi_\ell \,,
\end{split}
\end{align}
for all $x,y \in k$ or $\ell$ depending on whether the corresponding roots $r$ and $s$
are long or short, respectively.
Note that this list is exhaustive:
if $r$ and $s$ are long roots with $r+s \in \Phi$, then $r+s \in \Phi_\ell$;
and if $r$ is a short root and $s$ a long root with $r+s \in \Phi$, then $2r+s \in \Phi$ as well and $r+s$ is short and $2r+s$ is long.
See \cite[(1.2) and (1.3)]{Ree}.


\section{Mixed Chevalley groups of type $\F_4$}\label{se:mixedgroups}

Let $k$ and $\ell$ be fields of characteristic $2$ such that $\ell^2\leq k \leq \ell$.
Assume that $\delta \in k$ is such that the polynomial $x^2 + x + \delta$ is irreducible over $k$.
Let $\gamma$ be a solution of $x^2 + x = \delta$, and let $K = k(\gamma)$ and $L = \ell(\gamma)$.
Then $L^2 \leq K = \langle k, L^2 \rangle \leq L$,
and $K$ and $L$ are separable quadratic extensions of $k$ and $\ell$, respectively.
We denote the standard involution on both $L$ and $K$ corresponding to $\gamma$ by $x \mapsto \overline{x}$.

Let $\Phi$ be a root system of type $\F_4$ with fundamental system $\Pi:=\{\alpha_1,\alpha_2,\alpha_3,\alpha_4\}$. We can represent the fundamental roots with respect to an orthonormal basis $\{e_1,e_2,e_3,e_4\}$ of $\R^4$ as $\alpha_1=\frac{1}{2}(-e_1-e_2-e_3+e_4)$, $\alpha_2=e_3$, $\alpha_3=e_2-e_3$, $\alpha_4= e_1-e_2$ and the full system of roots is given by
\[
\Phi = \begin{cases}
	\pm e_i\pm e_j \text{ for } 1\leq i< j\leq 4 , \\
	\pm e_i \text{ for } 1\leq i\leq 4 , \\
	\frac{1}{2}(\pm e_1\pm e_2\pm e_3 \pm e_4) .
\end{cases}
\]

We define the mixed Chevalley group $\F_4(K,L)$ of type $\F_4$ as the mixed group that can be obtained from the ordinary Chevalley group $\F_4(L)$ of type $\F_4$.
For this, we remark (using the definitions introduced in section \ref{Chevalley}) that $H$ is a maximal $K$-split torus, $N=N_{\F_4(L)}(H)$ is the normalizer of $H$ in $\F_4(L)$ and $B$ is a Borel subgroup of $\F_4(L)$. Then
\[ \F_4(K,L)=\Big\langle \{u_r(s)\mid r\in\Phi_\ell, s\in K\} \cup \{u_r(t)\mid r\in\Phi_s,t\in L\} \cup T(K,L)\Big\rangle \]
is the mixed group of type $\F_4$ corresponding to the pair of fields $(K,L)$ of $\F_4(L)$.
Using the same procedure, we can construct mixed Chevalley groups of type $\B_n$, $\mathsf{C}_n$ and $\mathsf{G}_2$.
In general, we denote a mixed Chevalley group by $X(K,L)$.

Theorem~\ref{T(K,L)} below shows that we can omit the subgroup $T(K,L)$ in the generating set for $X(K,L)$ if $X(K,L)$ is of type $\mathsf{G}_2$, $\F_4$ or $\B_n$ with $n$ odd.
We will need the following observation.
\begin{lemma}\label{le:long}
	Let $\Phi$ be a root system of type $\B_n$, $\mathsf{C}_n$, $\F_4$ or $\mathsf{G}_2$, and let $p=3$ in the case of $\mathsf{G}_2$ and $p=2$ otherwise.
	Let $\Pi = \{ \alpha_1,\dots,\alpha_n \}$ be a set of fundamental roots for $\Phi$, and let $\Pi_s$ be the subset of $\Pi$ of short fundamental roots.
	If $r \in \Phi$ is a long root, and $r = \sum_{i=1}^n n_i \alpha_i$,
	then each coefficient $n_i$ corresponding to a short fundamental root $\alpha_i \in \Pi_s$ is divisible by $p$.
\end{lemma}
\begin{proof}
	This can easily be checked by a case by case analysis; see, for example, \cite[section 3.6]{Carter}.
\end{proof}

\begin{theorem}\label{T(K,L)}
	Let $X(K,L)$ be a mixed Chevalley group of type $\B_n$, $\mathsf{C}_n$, $\F_4$ or $\mathsf{G}_2$,
	with $L^p \subseteq K \subsetneq L$, where $p=3$ in the case of $\mathsf{G}_2$ and $p=2$ otherwise.
	Let
	\[ T(K,L) = \bigl\{ h \in T(L) \mid r(h) \in K \text{ for all } r \in \Phi_l \bigr\} \]
	as before.  Then
	\begin{equation}\label{eq:TKL}
		T(K,L) = \Big\langle \{h_r(t)\mid r\in\Phi_l,\ t\in K^\times\} \cup \{h_r(t)\mid r\in \Phi_s,\ t\in L^\times\} \Big\rangle
	\end{equation}
	if and only if $X(K,L)$ has type $\mathsf{G}_2$, $\F_4$ or $\B_n$ with $n$ odd.
	In this case, we have, in particular,
	\[ X(K,L)=\Big\langle \{u_r(s)\mid r\in\Phi_\ell, s\in K\} \cup \{u_r(t)\mid r\in\Phi_s,t\in L\} \Big\rangle. \]
\end{theorem}
\begin{proof}
Let $H=\langle h_r(\lambda)\mid r\in\Phi, \lambda\in L^\times \rangle$, i.e.\@ $H$ is the torus $T(L)$ of the Chevalley group $X(L)$
(as defined in section~\ref{Chevalley}).
Let $\Pi=\{\alpha_1,\dots, \alpha_n\}$ be the set of fundamental roots of $\Phi$,
and let $\Pi_s$ and $\Pi_l$ be the subsets of $\Pi$ of short and long fundamental roots, respectively.
We claim that
\begin{equation}\label{eq:h}
	H = T(L) = \prod_{i=1}^n h_{\alpha_i}(L^\times) .
\end{equation}
Indeed, if $r \in \Phi$ is any root, then $h_r \colon L^\times \to T(L)$ is precisely the coroot $r^\vee$ of $r$,
and hence we can write $r^\vee$ as an integral linear combination $r^\vee = \pm \sum_{i=1}^n n_i \alpha_i^\vee$ of
the coroots $\alpha_i^\vee$ corresponding to the roots $\alpha_i$.
Hence $h_r(t) = \prod_{i=1}^n h_{\alpha_i}(t^{\pm n_i})$ for all $t \in L^\times$,
and the claim~\eqref{eq:h} follows.

Notice that by the same argument, the equality~\eqref{eq:TKL} is equivalent with
\[ T(K,L) = \prod_{r \in \Pi_l} h_r(K^\times) \cdot \prod_{r \in \Pi_s} h_r(L^\times) , \]
which is, in view of \eqref{eq:h}, also equivalent with the implication
\begin{equation}\label{eq:TKL'}
	\prod_{r \in \Pi_l} h_r(\lambda_r) \in T(K,L) \implies \lambda_r \in K^\times \text{ for all } r \in \Pi_l ;
\end{equation}
so our goal is to show that the implication \eqref{eq:TKL'} holds if and only if $X(K,L)$ has type $\mathsf{G}_2$, $\F_4$ or $\B_n$ with $n$ odd.

Next, we claim that
\begin{equation}\label{eq:T}
	T(K,L) = \{ h \in T(L) \mid r(h) \in K \text{ for all } r \in \Pi_l \} .
\end{equation}
Indeed, assume that $h \in T(L)$ satisfies the condition that $r(h) \in K$ for all $r \in \Pi_l$,
and let $r \in \Phi_l$ be arbitrary.
Write $r = \sum_{i=1}^n n_i \alpha_i$;
hence $r(h) = \prod_{i=1}^n \alpha_i(h)^{n_i}$.
By Lemma~\ref{le:long}, each coefficient $n_i$ corresponding to a short fundamental root $\alpha_i \in \Pi_s$ is divisible by $p$.
If $\alpha_i$ is a long fundamental root, then $\alpha_i(h) \in K$ by assumption;
if $\alpha_i$ is a short fundamental root, then $\alpha_i(h)^{n_i} \in L^{n_i} \subseteq L^p \subseteq K$,
and we conclude that $r(h) \in K$, proving claim~\eqref{eq:T}.

Notice that for each of the types $\B_n$, $\mathsf{C}_n$, $\F_4$ or $\mathsf{G}_2$,
the subset $\Pi_l$ of long fundamental roots corresponds to a subdiagram of the Dynkin diagram of type $\mathsf{A}_m$, with $m$ equal to $n-1, 1, 2$ or $1$, respectively.
We will write $\Pi_l = \{ \alpha_1, \dots, \alpha_m \}$ accordingly, where we number the fundamental roots in the canonical way.
Hence we will rewrite an element $h = \prod_{r \in \Pi_l} h_r(\lambda_r)$ as $h = \prod_{i=1}^m \alpha_i^\vee(\lambda_i)$,
and by \eqref{eq:T}, such an $h$ belongs to $T(K,L)$ if and only if
\begin{equation}\label{eq:aa}
	\prod_{i=1}^m \alpha_j \bigl( \alpha_i^\vee(\lambda_i) \bigr) \in K \text{ for all } j \in \{ 1,\dots,m \} .
\end{equation}

We now do a case by case analysis.
\begin{compactitem}
    \item
	If $X(K,L)$ is of type $\mathsf{C}_n$, then $m=1$, and condition \eqref{eq:aa} says that $\alpha_1(\alpha_1^\vee(\lambda_1)) = \lambda_1^2 \in K$.
	This is satisfied for any element $\lambda_1 \in L$, so since $K \neq L$, the implication~\eqref{eq:TKL'} is false.
    \item
	If $X(K,L)$ is of type $\mathsf{G}_2$, then $m=1$, and condition \eqref{eq:aa} says that $\alpha_1(\alpha_1^\vee(\lambda_1)) = \lambda_1^2 \in K$.
	Since $L^3 \subseteq K$, this implies $\lambda_1 = \lambda_1^{-2} \lambda_1^3 \in K$, and hence the implication~\eqref{eq:TKL'} is true.
    \item
	Assume that $X(K,L)$ is of type $\mathsf{B}_3$ or of type $\mathsf{F}_4$.
	Then $m=2$, and condition \eqref{eq:aa} says that $\alpha_1(\alpha_1^\vee(\lambda_1)) \alpha_1(\alpha_2^\vee(\lambda_2)) = \lambda_1^2 \lambda_2^{-1} \in K$,
	and $\alpha_2(\alpha_1^\vee(\lambda_1)) \alpha_2(\alpha_2^\vee(\lambda_2)) = \lambda_1^{-1} \lambda_2^2 \in K$.
	Since $L^2 \subseteq K$, this is equivalent with $\lambda_1 \in K$ and $\lambda_2 \in K$, and hence the implication~\eqref{eq:TKL'} is true.
    \item
	Assume finally that $X(K,L)$ is of type $\mathsf{B}_n$ with $n \geq 4$.
	Then $m=n-1$, and condition \eqref{eq:aa} says that
	\[ \lambda_1^2 \lambda_2^{-1} \in K, \quad \lambda_{i-1}^{-1} \lambda_i^2 \lambda_{i+1}^{-1} \in K \text{ for } i \in \{ 2,\dots,m-1\},
		\quad \lambda_{m-1}^{-1} \lambda_m^2 \in K . \]
	Since $L^2 \subseteq K$, this is equivalent with
	\begin{gather*}
		\lambda_2 \in K, \quad \lambda_{m-1} \in K, \text{ and } \\
		\lambda_{i-1} \in K \iff \lambda_{i+1} \in K \ \text{ for } i \in \{ 2,\dots,m-1\} .
	\end{gather*}
	If $n$ is odd, then also $m-1$ is odd, and we can repeatedly apply the last equivalence to deduce that 
	$\lambda_2, \lambda_4, \dots, \lambda_{n-1} \in K$ and $\lambda_{m-1}, \lambda_{m-3}, \dots, \lambda_1 \in K$;
	hence the implication~\eqref{eq:TKL'} is true in this case.
	If $n$ is even, however, then condition~\eqref{eq:aa} is equivalent to the fact that $\lambda_i \in K$ for all even values of $i$,
	without any conditions on the other $\lambda_i$.
	Since $K \neq L$, it follows that the implication~\eqref{eq:TKL'} is false in this case.
    \qedhere
\end{compactitem}
\end{proof}

\begin{remark}\label{rem:referee}
The original version of this paper only contained a proof for the positive statement in the case of groups of type $\F_4$ and $\B_n$, $n$ odd,
and our proof was more involved.
The referee pointed out how we could simplify the proof, and simultaneously get a complete answer for all possible mixed Chevalley groups.
We thank him for sharing his insight with us.
\end{remark}

The previous lemma will allow us to transfer known facts about BN-pairs of (ordinary) Chevalley groups to BN-pairs of mixed Chevalley groups.
Indeed, when $X(K,L)$ is a mixed group, the subgroups
\begin{align*}
	N(K,L) &:= N(K) \; T(K,L) \quad \text{and} \\
	B(K,L) &:= \bigl\langle T(K,L) \cup \{ U_r(K) \mid r \in \Phi_l^+ \} \cup \{ U_r(L) \mid r \in \Phi_s^+ \} \bigr\rangle
\end{align*}
form a BN-pair for $X(K,L)$.
Using Theorem~\ref{T(K,L)}, we actually get
\[ B(K,L) = B(L)\cap X(K,L) \quad \text{and} \quad N(K,L) = N(L) \cap X(K,L) \]
if $X(K,L)$ is of the appropriate type, where $(B(L),N(L))$ is the natural BN-pair of $X(L)$.
This implies (using the general properties of a BN-pair) that
\[ X(K,L)=B(K,L) \; N(K,L) \; B(K,L) \]
and that all parabolic subgroups containing $B(K,L)$ are of the form
\begin{equation}\label{eq:PJ(K,L)}
	P_J(K,L) := B(K,L) \; N_J(K,L) \; B(K,L)=P_J\cap X(K,L).
\end{equation}
Notice that $N(K,L)/T(K,L)$ is also isomorphic to the Weyl group of $X(L)$.
So $N_J(K,L)$ is the preimage of $W_J$ under the canonical epimorphism from $N(K,L)$ to $W$.

We end this section with a unique decomposition lemma for mixed Chevalley groups.
\begin{lemma}\label{mixed}
Let $X(K,L)$ be a mixed Chevalley group of type $\F_4$ or of type $\B_n$ with $n$ odd. If $g\in X(K,L)$ is such that
\begin{equation}\label{eq:PgP}
	P_J(K,L)gP_J(K,L)=P_J(K,L)nP_J(K,L)
\end{equation}
with $n\in N(K,L)$ such that $nT(K,L)=w\in \operatorname{Stab}(\Phi_J)$, then $g$ has a unique decomposition $g=u ln u'$ with $u \in U_J(K,L):=U_J\cap X(K,L)$, $l\in L_J(K,L):= L_J\cap X(K,L)$ and $u'\in {U}_{w,J}^{-}$, where
\[ U_{w,J}^{-}:=\left\langle U_r\mid r\in \Phi^{+}\setminus \Phi_{J}, w(r)\in\Phi^{-}\right\rangle , \]
with
\[ U_r :=
\begin{cases}
	\{ u_r(s) \mid s\in K \} & \text{if } r \in \Phi_\ell, \\
	\{ u_r(t) \mid t\in L \} & \text{if } r \in \Phi_s.
\end{cases} \]
\end{lemma}
\begin{proof}
From the equality \eqref{eq:PgP}, we find that $g = p_1 n p_2$ for some $p_1,p_2\in P_J\cap X(K,L)$.
As ordinary Chevalley groups have a Levi decomposition $P_J=L_J\cdot U_J$, it follows, using~\eqref{eq:PJ(K,L)}, that there is a corresponding Levi decomposition
\[ P_J(K,L)= L_J(K,L)\cdot U_J(K,L). \]
Assume $p_2 = l'u'$ for some $l'\in L_J(K,L)$ and $u'\in U_J(K,L)$, then (as $nH\in\operatorname{Stab}(\Phi_J)$), we can switch $l'$ to the left of $n$;
moreover, we can also switch the factors of $u'$ belonging to some $U_r$ with $r \in \Phi^+ \setminus \Phi_J$ and $w(r) \in \Phi^+$, to the left of $n$,
so that we are left with an element $u' \in U_{w,J}^{-}$.
We find that indeed $g=p n u'=ulnu'$ for some $u\in U_J(K,L)$, $l\in L_J(K,L)$, and $u' \in U_{w,J}^{-}$.

Suppose that $g=u_1l_1nu'_1=u_2l_2nu'_2$, then $nu'_1{u'_2}^{-1}n^{-1}=(u_1l_1)^{-1}u_2l_2\in U_J^{-}\cap P_J=1$, so uniqueness follows.
\end{proof}


\section{Construction of a split BN-pair of rank one}\label{se:BN}

In this section we construct a split saturated BN-pair out of an involution on $\F_4(K,L)$. 
For Chevalley groups there exists a general procedure to construct a BN-pair from an involution $\sigma$ satisfying certain conditions,
as carried out in \cite{Steinbach1}.

Using a similar procedure, we show that we can construct a split saturated BN-pair of rank one from a suitable involution on $\F_4(K,L)$.
From a geometric point of view, we actually have constructed, starting from a mixed building of type $\F_4$, a new type of Moufang set;
these Moufang sets will be called {\em Moufang sets of mixed type $\F_4$}.

\subsection{Construction of an involution on $\F_4(K,L)$}\label{Involution}

We follow the ideas from \cite{Steinbach1}, but in order to deal with the situation of mixed Chevalley groups,
we impose slightly adjusted conditions on the involution $\sigma$ on $\F_4(K,L)$.
More precisely, we fix a set $J\subsetneq \Pi$, and we choose $\sigma$ in such a way that 
\begin{compactenum}[(1)]
\item $\sigma$ permutes root groups and $N(K,L)$ is invariant under $\sigma$.
\item If $P$ is a parabolic subgroup of $L_J(K,L)=L_J(L)\cap \F_4(K,L)$, invariant under $\sigma$, then $P=L_J(K,L)$.\label{sigma}
\item $\langle \{U_r(K)\mid r\in\Phi_l^{-}\setminus \Phi_J\}\cup \{U_r(L)\mid r\in\Phi_s^{-}\setminus \Phi_J\}\rangle\cap \Fix(\sigma)\neq 1$.
\end{compactenum} 

In order to take care of the first condition, we consider an involution $\sigma$ of $\F_4(K,L)$ with the following action on the generators of the mixed group
(where we denote the corresponding action on the root system also by $\sigma$):
\begin{align*}
\sigma: \F_4(K,L)\to \F_4(K,L) \colon u_r(t)\mapsto u_{\sigma(r)}(c_r\,\overline{t}).
\end{align*}
In analogy with the situation in the algebraic case, we will choose the action of $\sigma$ on the root system so that
the corresponding Tits index is as follows:
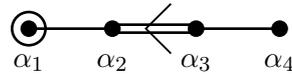
\begin{equation*} 
\begin{tikzpicture}[line width=1pt, scale=1.1]
	\draw (0,0) -- (1,0);
	\draw[double distance=2pt] (1,0) -- (2,0);
	\draw (2,0) -- (3,0);
	\draw[line width=.7pt] (1.7,.3) -- (1.4,0) -- (1.7,-.3);
	\diagnode{(0,0)}
	\diagnode{(1,0)}
	\diagnode{(2,0)}
	\diagnode{(3,0)}
	\draw[below=5pt] (0,0) node {$\alpha_1$};
	\draw[below=5pt] (1,0) node {$\alpha_2$};
	\draw[below=5pt] (2,0) node {$\alpha_3$};
	\draw[below=5pt] (3,0) node {$\alpha_4$};
	\distorbit{(0,0)}
\end{tikzpicture}
\end{equation*}
Notice that this is the only admissable Tits index of relative rank one of absolute type~$F_4$, and in fact,
every linear algebraic group of absolute type~$F_4$ is either anisotropic (i.e.\@ has relative rank~$0$),
or split (i.e.\@ has relative rank~$4$), or has the above Tits index.
(In the mixed case, however, an additional Tits index of relative rank~$2$ can arise; see \cite{MVM}.)

If we now look at the $\F_4$-building corresponding to $\F_4(K,L)$, our goal is to construct the involution $\sigma$ on $\F_4(K,L)$ in such a way that the corresponding
fixbuilding has only points of the first type (these are the points corresponding to $\alpha_1$).
Therefore, we choose $J$ to be the subset $\{\alpha_2,\alpha_3,\alpha_4\}$ of $\Pi$.
In particular, the action of $\sigma$ on the root system $\Phi$ is given by the longest element ${w_0}^J$ in the Weyl group $W_J$ generated by $w_{\alpha_2}$, $w_{\alpha_3}$ and $w_{\alpha_4}$.
Then $\sigma$ is an involution fixing $e_4$ and inverting $e_1$, $e_2$ and $e_3$,
implying that the action of $\sigma$ on $\Pi$ is given by
\[
\begin{cases}
\alpha_1 \mapsto \frac{1}{2}(e_1+e_2+e_3+e_4) = \alpha_1 + 3\alpha_2 + 2\alpha_3 + \alpha_4\\
\alpha_2\mapsto -\alpha_2\\
\alpha_3\mapsto -\alpha_3\\
\alpha_4\mapsto -\alpha_4.
\end{cases}
\]
Our next step will be to determine the coefficients $c_r$ so that $\sigma$ does indeed give rise to a Moufang set.
We will first focus on the second condition for~$\sigma$;
since this condition only concerns the subgroup $L_J(K,L)$ of $\F_4(K,L)$, we will achieve this by looking at the subgroup $\B_3(K,L) \leq \F_4(K,L)$ (this is the subgroup of $\F_4(K,L)$ generated by the root groups $U_{\alpha_2}(L)$, $U_{\alpha_3}(K)$ and $U_{\alpha_4}(K)$).
Once we will have constructed an involution $\sigma$ such that the second condition is satisfied,
we will see it is not very hard to check that also the third condition for $\sigma$ holds.

In the non-mixed case, we know that the action of $\sigma$ on the $\B_3$-building has to be chosen in such a way that the group fixed under $\sigma$ is isomorphic to the projective orthogonal group of an anistropic quadratic form of dimension $7$ with trivial Hasse invariant; see \cite[Section 3.4]{Selbach} for more details.
One can show that every such a quadratic form can be obtained as the restriction to the trace zero part of an $8$-dimensional norm form of an octonion division algebra (i.e.\@ of a $3$-fold Pfister form).
In a completely similar way, we obtain that the fixed point set of the involution $\sigma$ on the mixed $\B_3$-subbuilding has to be isomorphic to $\PGO(q)$ with $q$ the trace zero part of the `mixed' norm form of an octonion division algebra.
In the next subsection, we determine explicitly what this action should be and deduce in this way the coefficients~$c_r$.

\subsection{The action on the $\B_3$-subbuilding}

As we have seen in Example~\ref{ex:mixedq},
we can identify $\B_3(K,L)$ with the group $\PGO(R)$ where $R$ is the mixed quadratic form
\begin{multline*}
R:L\oplus K^6\to K; (x_0,x_1,x_{-1},x_2,x_{-2},x_3,x_{-3}) \\ \mapsto {x_0}^2+x_1x_{-1}+x_2x_{-2}+x_3x_{-3}
\end{multline*}
with respect to a well chosen hyperbolic basis $\mathcal{C}$;
this group consists of all $(K,L)$-linear maps $\varphi$ (modulo scalars) such that $R(\varphi(v))=R(v)$ for all $v\in L\oplus K^6$.
There is a bijective correspondence between $(K,L)$-linear maps $\varphi$ and the invertible $7$ by $7$ matrices $A$
such that the first row%
\footnote{We will always use {\em left} multiplication by matrices on column spaces.}
consists of elements in $L$,
while the others consist of elements in $K$ and all the elements in the first column, except the first one, are zero.
The condition $R(\varphi(v))=R(v)$ translates into $[R]^tA[R]=A$ where $[R]$ is the matrix%
\footnote{When we talk about {\em the} matrix of a quadratic form, we mean the unique upper-triangular matrix representing this quadratic form w.r.t.\@ the given basis.}
corresponding to the quadratic form $R$. 

First, we will determine the isomorphism between $\B_3(K,L)$ and $\PGO(R)$ explicitly, because this will allow us to describe the action of $\sigma$ on the $\B_3$-subbuilding entirely in terms of matrices. 

\subsubsection{Construction of an isomorphism between $\B_3(K,L)$ and $\PGO(R)$}\label{Isomorphism}

We use the correspondence between $\B_3(L)$ and $\PGO(\tilde{R})$ (with $\tilde{R}$ the unique extension of $R$ to a quadratic form on $L^7$)
mentioned in \cite[Section 11.3]{Carter} to determine a matrix representation for the elements of $\B_3(K,L)$.
Therefore we return to the original definition of $\B_3(L)$ as the group generated by some automorphisms on $L\otimes \mathcal{L}_{\Z}$.
We know $\mathcal{L}:=\mathcal{L}_{\C}$ has a Cartan decomposition
\[ \lie = \cartan \oplus \bigoplus_{r \in \Phi_{\B_3}} \lie_r =\cartan\oplus\bigoplus\C e_r, \]
where $e_r$ runs through the list (\cite[p.\@~180]{Carter})
\begin{align*}
&E_{i,j}-E_{-j,-i}
&&E_{i,-j}-E_{j,-i}
&&2E_{i,0}-E_{0,-i}\\
-&E_{-i,-j}+E_{j,i}
&-&E_{-i,j}+E_{-j,i}
&-&2E_{-i,0}+E_{0,i}
\end{align*}
for $0<i<j\leq 3$.
The matrices $E_{i,j}$ are the $7$ by $7$ matrices with a $1$ on the $(i,j)$-th position,
with rows and columns indexed by $\{0,1,-1,2,-2,3,-3\}$.

Next, we want to find an explicit correspondence between the roots of $\Phi_{\B_3}$ and the root spaces $\C e_r$ of $\lie$.
Therefore, it suffices to find an identification between the fundamental roots of $\mathcal{H}^*$ and those of $\Phi_{\B_3}$.
According to \cite[p.\@~180]{Carter}, the elements of $\mathcal{H}$ are of the form $\diag(0,\lambda_1,-\lambda_1,\lambda_2,-\lambda_2, \lambda_3,-\lambda_3)$, $\lambda_i\in\C$.
We find a fundamental system $\{\tilde{\alpha}_2,\tilde{\alpha}_3,\tilde{\alpha}_4\}$ for $\mathcal{H}^*$ with
\begin{alignat*}{3}
\tilde{\alpha}_2 \colon &\mathcal{H}\to \C; &&\diag(0,\lambda_1,-\lambda_1,\lambda_2,-\lambda_2, \lambda_3,-\lambda_3) &&\mapsto \lambda_3 , \\
\tilde{\alpha}_3 \colon &\mathcal{H}\to \C; &&\diag(0,\lambda_1,-\lambda_1,\lambda_2,-\lambda_2, \lambda_3,-\lambda_3) &&\mapsto \lambda_2-\lambda_3 , \\
\tilde{\alpha}_4 \colon &\mathcal{H}\to \C; &&\diag(0,\lambda_1,-\lambda_1,\lambda_2,-\lambda_2, \lambda_3,-\lambda_3) &&\mapsto \lambda_1-\lambda_2 .
\end{alignat*}
With the use of equation \eqref{H}, we obtain that the elements $\tilde{\alpha}_2$, $\tilde{\alpha}_3$, $\tilde{\alpha}_4$ of $\mathcal{H}^*$ correspond to the elements $2E_{3,0}-E_{0,-3}$, $E_{2,3}-E_{-3,-2}$ and $E_{1,2}-E_{-2,-1}$ of the Chevalley basis, respectively. The obvious identification one can make between the roots of $\mathcal{H}^*$ and those of $\Phi_{\B_3}$ is
\begin{align*}
\tilde{\alpha}_2 &\ \leftrightarrow\ \alpha_2=e_3\\
\tilde{\alpha}_3 &\ \leftrightarrow\ \alpha_3=e_2-e_3\\
\tilde{\alpha}_4 &\ \leftrightarrow\ \alpha_4=e_1-e_2.
\end{align*}
We can now identify the elements $u_r(t)$ with matrices. This can be done using the epimorphism
\[G\to \lie(L): \ex(t e_r )\mapsto u_r(t),\]
with $\ex(t e_r)$ being the matrices described on \cite[p.183]{Carter} and $G$ being the group generated by all these matrices. The kernel of this map turns out to be the center of $G$.

In this way we can identify the following elements for all $i,j\in \{1,\dots, 3\}$: 
\begin{align*}
 u_{e_i-e_j}(\lambda) &\ \leftrightarrow\ \I+\lambda (E_{i,j}+E_{-j,-i})\\
 u_{e_i+e_j}{(\lambda)} &\ \leftrightarrow\ \I+\lambda (E_{i,-j}+E_{j,-i})\\
 u_{-e_i-e_j}{(\lambda)} &\ \leftrightarrow\ \I+\lambda (E_{-i,j}+E_{-j,i})\\
 u_{e_i}(\lambda) &\ \leftrightarrow\ \I + \lambda E_{0,-i} +\lambda^2\, E_{i,-i} \\
 u_{-e_i}(\lambda) &\ \leftrightarrow\ \I + \lambda E_{0,i} +\lambda^2\, E_{-i,i}.
\end{align*}
for all $\lambda \in L$, where $\I$ is the $7$ by $7$ identity matrix.

\subsubsection{The action of $\sigma$ on $\PGO(R)$}

As mentioned in the previous subsection, we wish to construct $\sigma$ in such a way that the fixed points form a group isomorphic to $\PGO(q)$  with $q$,
the trace zero part of a mixed norm form of an octonion division algebra. 
Such a quadratic form is defined over the fields $k$ and $\ell$ and is of the form
\begin{multline*}
\N := \N_L\perp \alpha \N_K \perp \beta \N_K \perp \alpha\beta \N_K \colon \\
\begin{aligned}
	L\oplus K \oplus K \oplus K &\to \ell; \\
	(y_1,y_2,y_3,y_4) &\mapsto y_1\overline{y_1}+\alpha y_2\overline{y_2}+\beta y_3\overline{y_3}+\alpha\beta y_4\overline{y_4} \,,
\end{aligned}
\end{multline*}
where $\alpha, \beta$ are constants in $k^\times$.
\begin{remark}\label{re:octmult}
	Denote by  $\N_\ell$ the extension of $\N$ to the octonion algebra $\mathcal{O}_\ell=L\oplus L\oplus L\oplus L$.
	Although the norm on $\mathcal{O}_\ell$ is uniquely determined, there is no canonical way to define the product of two octonions
	(in terms of the decomposition $\oct_\ell = L^4$),
	although all of these algebras are isomorphic.
	The most common way to define such a multiplication uses the fact that every composition algebra of dimension $d>1$ can be obtained from a $(d/2)$-dimensional subalgebra
	by the so-called Cayley--Dickson doubling process;
	see, for example, \cite[Proposition 1.5.3]{octonions}.
	For our purposes, it will be more convenient to use the following description.

	Let $x=(x_1,x_2,x_3,x_4)$ and $y=(y_1,y_2,y_3,y_4)$ be two arbitrary elements of $\oct_\ell$, then we define the product $x\cdot y$ to be equal to
	\begin{multline*}
		\bigl( x_1y_1+\alpha\overline{x_2}y_2+\beta\overline{x_3}y_3+\alpha\beta\overline{x_4}y_4, \ x_2y_1+\overline{x_1}y_2+\beta x_4\overline{y_3}+\beta\overline{x_3}y_4, \\
		x_3y_1+\overline{x_1}y_3+\alpha x_4\overline{y_2}+\alpha \overline{x_2}y_4,\ x_3y_2+x_2y_3+x_4\overline{y_1}+x_1y_4 \bigr) ,
	\end{multline*}
	and we define the conjugate $\overline{x}$ to be equal to
	\[ \overline{x} = \bigl( \overline{x_1}, x_2, x_3, x_4 \bigr) ; \]
	this makes $\mathcal{O}_\ell = L^4$ into an octonion algebra with norm $N_\ell$, and $N_\ell(x) = x \cdot \overline{x} = \overline{x} \cdot x$
	for all $x \in \mathcal{O}_\ell$.
\end{remark}

The restriction $q$ of $\N$ to the subspace of trace zero elements is then of the following form:
\begin{align*}
q: \ell\oplus K\oplus K\oplus K \to k; (y_1,y_2,y_3,y_4)\mapsto y_1^2+\alpha y_2\overline{y_2}+\beta y_3\overline{y_3}+\alpha\beta y_4\overline{y_4}.
\end{align*}
We will now show how to construct such an involution $\sigma$.

Viewing the mixed quadratic form $q$ as a form over $(K,L)$ in the obvious way (we denote this extended form by $Q := q\otimes_{k,\ell} (K,L)$),
the quadratic forms $Q$ and $R$ are isometric.
Indeed, looking at the matrix representations $[Q]_\mathcal{B}$ and $[R]_\mathcal{C}$ with respect to the standard bases $\mathcal{B}$ and $\mathcal{C}$
we see that
\begin{align*}
[Q]_\mathcal{B}=\begin{pmatrix}
1 &  & &  & & &\\
& \alpha & \alpha & & & \\
 & 0 & \alpha\delta \\
 &  && \beta & \beta\\
 & &&0 & \beta\delta\\
 & &&&& \alpha\beta & \alpha\beta\\
& &&&& 0 & \alpha\beta\delta
\end{pmatrix}\ \text{and}\
[R]_\mathcal{C}=\begin{pmatrix}
1 &  & &  & & &\\
& 0 & 1 & & & \\
 & 0 & 0 \\
 &  && 0 & 1\\
 & &&0 & 0\\
 & &&&& 0 & 1\\
& &&&& 0 & 0
\end{pmatrix}.
\end{align*}
A change of the base of $\mathcal{B}$ using the transition matrix
\begin{align*}S:=\begin{pmatrix}
1 &  & &  & & &\\
& 1 & \gamma & & & \\
 & \alpha & \alpha\overline{\gamma} \\
 &  && 1 & \gamma\\
 & &&\beta & \beta\overline{\gamma}\\
 & &&&& 1 & \gamma\\
& &&&& \alpha\beta & \alpha\beta\overline{\gamma}
\end{pmatrix}
\end{align*}
will do the job.
(Notice that the matrix $S^t [R]_\mathcal{C} S$ is not equal to the matrix $[Q]_\mathcal{B}$, but it represents the same quadratic form.)

It follows that $\PGO(Q)^{S^{-1}} = \PGO(R)$, and since $\PGO(q)$ is equal to $\{A\in\PGO(Q)\mid A=\overline{A}\}$ we obtain
\begin{align*}
\PGO(q)^{S^{-1}} &=\{SAS^{-1}\mid A\in \PGO(Q)\ \text{and}\ A=\overline{A} \}\\
&= \{B\in \PGO(R)\mid S^{-1}BS=\overline{S^{-1}BS}\}\\
&=\{B\in \PGO(R)\mid B=(S\overline{S}^{-1})\overline{B}(\overline{S}S^{-1})\}\\
&=\{B\in \PGO(R)\mid B=M^{-1}\overline{B}M\}
\end{align*}
with 
\begin{align*}
M=\overline{S}S^{-1}
=\begin{pmatrix}
1 &  & &  & & &\\
& 0 & \alpha^{-1} & & & \\
 & \alpha & 0\\
 &  && 0 & \beta^{-1}\\
 & &&\beta & 0\\
 & &&&& 0 & \alpha^{-1}\beta^{-1}\\
& &&&& \alpha\beta & 0
\end{pmatrix}.
\end{align*}

We conclude that we can describe the restriction of the involution $\sigma$ (using the isomorphism between $\PGO(R)$ and $\B_3(K,L)$) as
\begin{equation}\label{eq:sigmaB3}
	\sigma_{|_{\B_3(K,L)}}: \PGO(R)\to \PGO(R); x\mapsto M^{-1}\overline{x}M.
\end{equation}
We will from now on identify $\PGO(R)$ and $\B_3(K,L)$ without explicitly mentioning the isomorphism.

\subsubsection{Calculation of the coefficients $c_r$}\label{ss:cr}

Using the identification between $\PGO(R)$ and $\B_3(K,L)$, we can now write the involution $\sigma$ as
\begin{align*}
\sigma(u_r(t))=M^{-1}\overline{u_r(t)}M=u_{\sigma(r)}(c_r\overline{t}).
\end{align*}
for all $r\in \Phi_{\B_3}$. 
Using~\eqref{eq:sigmaB3} we find for the generators $\alpha_2,\alpha_3$ and $\alpha_4$ of $\Phi_{\B_3}$ that $c_{\alpha_2}=\alpha\beta$, $c_{\alpha_3}=\alpha^{-1}$ and $c_{\alpha_4}=\alpha\beta^{-1}$.

It remains to determine the coefficient $c_{\alpha_1}$ because then all $c_r$ follow using the Chevalley commutator relations. Since the anisotropic subbuilding is of the right form, the only thing we still have to express is that $\sigma$ should be an involution. This is fulfilled if $\overline{c_{\alpha_1}}c_{\sigma(\alpha_1)}=1$.

We would like to deduce all coefficients $c_r$ with $r\in\Phi^+$ arbitrary and consequently $c_{\sigma(\alpha_1)}$.
By applying $\sigma$ on the non-trivial relations from~\eqref{eq:comm2}, we see that
\begin{alignat*}{2}
	&c_r c_s = c_{r+s} &&\text{ when } r,s \in \Phi_s \text{ and } r+s \in \Phi_s \,, \\
	&c_r c_s = c_{r+s} &&\text{ when } r,s \in \Phi_\ell \text{ and } r+s \in \Phi_\ell \,, \\
	&c_r c_s = c_{r+s} \text{ and } c_r^2 c_s = c_{2r+s} \ &&\text{ when }
		r \in \Phi_s, s \in \Phi_\ell \text{ and } r+s \in \Phi_s \,.
\end{alignat*}
Now let $r$ be any positive root, and write $r=\sum_i\lambda_i\alpha_i$, with all $\lambda_i$ positive integers.
By \cite[Lemma 3.6.2]{Carter} (or \cite[Chapter VI, section 1.6, Proposition~19]{BourLie}), we can obtain $r$ by adding one fundamental root $\alpha_i$ at a time,
and therefore we inductively obtain $c_r=\prod{c^{\lambda_i}_{\alpha_i}}$.
When we apply this on $\sigma(\alpha_1) = \alpha_1 + 3\alpha_2 + 2\alpha_3 + \alpha_4$,
we find that $\overline{c_{\alpha_1}}c_{\sigma(\alpha_1)}=1$ for $c_{\alpha_1}=\alpha^{-1}\beta^{-1}$.

The other coefficients (belonging to negative roots) can be found using the relation $\overline{c_{r}}c_{\sigma(r)}=1$ for every $r\in\Phi$. As before, this relation follows from the fact that $\sigma$ is an involution.

\subsection{Description of the split saturated BN-pair }\label{multiplication}

We show in this section that $G^1=\langle U^1,V^1\rangle$ with 
\begin{align*}
U^1:= U_J \cap\Fix(\sigma)\\
V^1:=U_J^{-}\cap\Fix(\sigma)
\end{align*}
has a split saturated BN-pair of rank one. 

We verify that $\sigma$ satisfies condition (2) on page~\pageref{sigma}:
\begin{lemma}\label{le:sigma2}
No parabolic subgroups of $L_J(K,L)$ are fixed.
\end{lemma}
\begin{proof}
We prove that no parabolic subgroups of $\B_3(K,L)$ are fixed. This is enough since if $L_J(K,L)$ has a parabolic subgroup $P$ fixed by $\sigma$, then $P\cap\B_3(K,L)$ is a fixed parabolic subgroup of $\B_3(K,L)$.
%

We have a closer look at the building corresponding to a general group $G$ with BN-pair $(B,N)$.
By \cite[Section 6.2.6]{AbramenkoBrown}, the parabolic subgroups of $G$ ordered by the opposite of the inclusion relation form the simplicial complex of the building. In particular, the chambers (i.e.\@ maximal flags) correspond to conjugates under $G$ of $B$. Also, the group $B$ is exactly the stabilizer in $G$ of the chamber corresponding to $B$. Furthermore, the parabolic subgroups $P_J=BN_JB$ are exactly the stabilizers in $G$ of their corresponding flags.

The group $\B_3(K,L)$ has a BN-pair $(\B_3(K,L)\cap B(K,L), \B_3(K,L)\cap N(K,L))$. Since $\PGO(R)$ is isomorphic to $\B_3(K,L)$, we know from the theory of buildings that the building we obtain is the mixed quadric corresponding to $R$. More specifically, the quadric corresponding to $R$ has points, lines and planes (since $R$ has Witt index 3). The flags of this quadric are then exactly the flags of the building of $\PGO(R)$. So conjugates of $\B_3(K,L)\cap B(K,L)$ correspond to triples $(p,L,\pi)$ with $p$ a point of $L$ and $L$ a line on the plane $\pi$, while maximal parabolic subgroups correspond to points, lines or planes of the mixed quadric.
From now on, we will assume that $(p,L,\pi)$ is the maximal flag corresponding to the standard minimal parabolic subgroup $\B_3(K,L)\cap B(K,L)$.

Let $\mathcal{C}=( x, x_1, y_1, x_2, y_2, x_3, y_3)$ be a hyperbolic basis for the mixed quadratic form $R$,
i.e.\@ a basis for the $K$-vectorspace $L\oplus K^6$ such that
\[ R(x)=1, \langle x,x\rangle=\langle x,x_i\rangle=\langle x,y_i\rangle=\langle x_i,x_j\rangle =\langle y_i,y_j\rangle =0, \langle x_i, y_j\rangle =\delta_{ij}, \]
where $\langle \cdot , \cdot \rangle$ is the bilinear form corresponding to $R$.

We observe that $(\langle x_1\rangle, \langle x_1,x_2\rangle, \langle x_1,x_2,x_3\rangle)$ forms a chamber in the building of $\PGO(R)$.
We claim that this chamber is precisely the chamber $(p,L,\pi)$ corresponding to the standard minimal parabolic $\B_3(K,L) \cap B(K,L)$
under the isomorphism between $\B_3(K,L)$ and the matrix group corresponding to $\PGO(R)$ constructed in section~\ref{Isomorphism}.
To prove our claim, we have to show that all generators of $\B_3(K,L)\cap B(K,L)$ fix the subspaces $\langle x_1\rangle, \langle x_1,x_2\rangle$ and $\langle x_1,x_2,x_3\rangle$.
Consider the generators of the form $x_{e_1-e_2}(t)$ with $t\in K$. These elements correspond
to matrices $A=\I+t(E_{1,2}+E_{-2,-1})$, so we need to check that
\begin{align*}
&A(0,\lambda_1,0,0,0,0,0)^t\in \langle x_1\rangle\\
&A(0,\lambda_1,0,\lambda_2,0,0,0)^t\in \langle x_1,x_2 \rangle\\
&A(0,\lambda_1,0,\lambda_2,0,\lambda_3,0)^t\in \langle x_1,x_2,x_3\rangle
\end{align*}
for all $\lambda_1,\lambda_2, \lambda_3\in K$, which is easily verified.
The other generators can be treated similarly, and this proves our claim.

Now suppose that a parabolic subgroup S of $\B_3(K,L)$ is fixed by $\sigma$; our goal is to derive a contradiction.
As $\sigma$ is type-preserving, we know that if a flag is fixed, then certainly a point, line or plane must be fixed.
As all parabolic subgroups are conjugate, there is some $g \in \B_3(K,L)$ such that $S=P^g$,
where $P$ is a standard parabolic subgroup, i.e.\@ $P$ contains $(\B_3(K,L)\cap B(K,L))$.
So $S$ corresponds to a flag contained in the chamber $(g(p),g(L),g(\pi))$, and hence one of the subspaces $\langle g(x_1)\rangle$, $\langle g(x_1),g(x_2)\rangle$ or $\langle g(x_1),g(x_2),g(x_3) \rangle$ has to be fixed under $\sigma$. 

We claim that the involution $\sigma$ maps $g(p)$, $g(L)$ and $g(\pi)$ to $\langle\sigma(g)(y_1)\rangle$, $\langle\sigma(g)(y_1),\sigma(g)(y_2)\rangle$ and $\langle\sigma(g)(y_1),\sigma(g)(y_2),\sigma(g)(y_3)\rangle$, respectively.
From this we deduce that in each of the $3$ cases (fixed point, line or plane),
$\sigma(g)(y_1)\in \langle g(x_1), g(x_2), g(x_3)\rangle$.
In particular, $\langle g(x_1)\rangle\perp \langle \sigma(g)(y_1)\rangle$. 

In order to prove our claim, we need to show that $\sigma(B^g)$, where $B^g$ is the stabilizer of the flag 
\[ \bigl( g(\langle x_1\rangle), \; g(\langle x_1,x_2\rangle), \; g(\langle x_1,x_2,x_3\rangle) \bigr) ,\] 
fixes the flag 
\[ \bigl( \langle\sigma(g)(y_1)\rangle, \; \langle\sigma(g)(y_1),\sigma(g)(y_2)\rangle, \; \langle\sigma(g)(y_1),\sigma(g)(y_2),\sigma(g)(y_3)\rangle \bigr) .\]
This is equivalent to showing that
$\sigma(B)$ is the stabilizer of the flag
\[ \bigl( \langle y_1\rangle,\langle y_1,y_2\rangle,\langle y_1,y_2,y_3\rangle \bigr) . \]
This last statement can again easily be checked on each of the generators of~$B$, and this proves our claim.

Suppose next that $g(x_1)=(z, z_1,a_1,z_2,a_2,z_3,a_3)$, then $z^2+\sum_i z_ia_i=0$ since $R(g(x_1))=R(x_1)=0$.
Notice that $g(x_1)$ is the second column of the matrix corresponding to $g$ and that $\sigma(g)(y_1)=M^{-1}\overline{g}M(y_1)$ is the third column of the matrix $M^{-1}\overline{g}M$;
this implies, using the explicit description of the matrix $M$, that
\[ \sigma(g)(y_1)=\alpha^{-1} \bigl( \overline{z},\alpha^{-1}\overline{a_1},\alpha\overline{z_1},\beta^{-1}\overline{a_2},\beta\overline{z_2},\alpha^{-1}\beta^{-1}\overline{a_3},\alpha\beta\overline{z_3} \bigr) . \]
Since $\langle g(x_1)\rangle\perp \langle \sigma(g)(y_1)\rangle$, we get
\[ \alpha z_1\overline{z_1}+\alpha^{-1}a_1\overline{a_1}+\beta z_2 \overline{z_2}+\beta^{-1}a_2\overline{a_2}+\alpha\beta z_3\overline{z_3}+\alpha^{-1}\beta^{-1}a_3\overline{a_3}=0. \]
This is equivalent with
\begin{multline*}
	(z+\overline{z})^2+\alpha (\overline{z_1}+\alpha^{-1}a_1)(z_1+\alpha^{-1}\overline{a_1})+\beta (\overline{z_2}+\beta^{-1}a_2)(z_2+\beta^{-1}\overline{a_2}) \\
	+\alpha\beta(\overline{z_3}+\alpha^{-1}\beta^{-1}a_3)(z_3+\alpha^{-1}\beta^{-1}\overline{a_3})=0 .
\end{multline*}
Since $q$ is anistropic, this implies $a_1=\alpha\overline{z_1}$, $a_2=\beta\overline{z_2}$ and $a_3=\alpha\beta\overline{z_3}$. Finally, by expressing again that $R(g(x_1))=0$, we obtain that
\[ z^2+\alpha z_1\overline{z_1}+\beta z_2\overline{z_2}+\alpha\beta z_3\overline{z_3}=0 , \]
and hence $z=0$ and $z_i=0$ for all $i$, a contradiction.
We conclude that no parabolic subgroup of $\B_3(K,L)$ is fixed.
\end{proof}

To proceed, we assemble a few lemmas about mixed BN-pairs.
We write $W$ for the Weyl group $N(K,L)/T(K,L)$, which is isomorphic to the Weyl group corresponding to a root system of type $\F_4$,
and we use the notation $C_W(\sigma)$ for the centralizer in $W$ of $\sigma$,
where we identify $\sigma$ with the corresponding element ${w_0}^J$ in the Weyl group $W_J$ generated by $w_{\alpha_2}$, $w_{\alpha_3}$ and $w_{\alpha_4}$
(see section~\ref{Involution}).
\begin{lemma}
Let $g\in \F_4(K,L)$ such that $\sigma(g)\in P_J(K,L)gP_J(K,L)$. If $P_J(K,L)gP_J(K,L)=P_J(K,L)nP_J(K,L)$ for some $n\in N(K,L)$ corresponding to the shortest element $w$ in $W_JwW_J$. Then $w\in C_W{(\sigma)}$.
\end{lemma}
\begin{proof}
	See \cite[Lemma 2.4]{Steinbach1}.
	Although the proof is not stated for mixed Chevalley groups, it can be copied almost verbatim,
	by replacing $P_J$ and $H$ by $P_J(K,L)$ and $T(K,L)$, respectively.
\end{proof}
The next lemma is a mixed version of \cite[Lemma 2.5]{Steinbach1}.
We notice that in the proof of this lemma we need the assumption (2) made in Section~\ref{Involution}, page~\pageref{sigma}, which we proved in Lemma~\ref{le:sigma2} above.
\begin{lemma}\label{Stab}
Let $g\in \F_4(K,L)$ with $^g\!P_J(K,L)=gP_J(K,L)g^{-1}$ invariant under $\sigma$. If $P_J(K,L)gP_J(K,L)=P_J(K,L)nP_J(K,L)$, with $n\in N(K,L)$ such that the corresponding element $w$ of $W$ is the shortest element in $W_JwW_J$, then $w\in C_W(\sigma)\cap \Stab(\Phi_J)$.
\end{lemma}
\begin{proof}
Let $g=pnp'$ with $n\in N(K,L)$ and $p,p'\in P_J(K,L)$.
Let $I:=J\cap w(J)$; then
\[ P_I(K,L) = U_J(K,L)(P_J(K,L)\cap {^n\!P_J(K,L)}) . \]
Hence $^p\!P_I(K,L)=U_J(K,L)(P_J(K,L)\cap {^g\!P_J(K,L)})$ is $\sigma$-invariant.
Furthermore, if $p=lu$ with $l\in L_J(K,L)$ and $u\in U_J(K,L)$, then
\[ ^l(L_J\cap P_I(K,L))=L_J(K,L)\cap {^p\!P_I(K,L)} \]
is a parabolic subgroup of $L_J(K,L)$.
By Lemma~\ref{le:sigma2}, $L_J(K,L)\cap {\,^p\!P_I(K,L)}=L_J(K,L)$.
We conclude that $P_J(K,L)\subseteq P_I(K,L)$, so $J=I$ and therefore $w(J)=J$.
\end{proof}
\begin{lemma}\label{Product}
Let $1\neq w \in W$ with $w(\Phi_J^+)=\Phi_J^+$. Then $w_0^Jw = w_0$.
\end{lemma}
\begin{proof}
See \cite[Lemma 2.6]{Steinbach1}.
\end{proof}

In the next paragraph, we will prove that $B^1:=P_J(K,L)\cap G^1$, together with a suitable $N^1$ (which we will construct on the way) forms a split saturated BN-pair for $G^1$.
We let $H^1:=L_J(K,L)\cap G^1$.

We use the proof of \cite[Lemma 2.7]{Steinbach1} to construct an element $\tilde{n}\in (n_0L_J)\cap G^1$ with $n_0$ an arbitrary element of $N(K,L)$ such that $n_0T(K,L)=w_0$, the longest element in $W=N(K,L)/T(K,L)$. 
\begin{lemma}
Let $n_0\in N(K,L)$ be such that $n_0T(K,L)=w_0$, then $n_{e_4}\in n_0L_J(K,L)\cap G^1$.
\end{lemma}
\begin{proof}
We notice that $u_{-e_4}(1)\in G^1\cap U_{\Phi^{-}\setminus \Phi_J^{-}}$. Furthermore, $P_Ju_{-e_4}(1)P_J=P_Jn_{e_4}P_J$ with $w_{e_4}$ the shortest element in $W_Jw_{e_4}W_J$. Lemma~\ref{Stab} shows that $w_{e_4}\in C_W(\sigma)\cap \Stab(\Phi_J)$.
This together with Lemma \ref{Product} allows us to conclude that $w_0=w_0^{J}w_{e_4}$, with $w_{0}^J$ being the longest element in $\Phi_J$.

Since $W \cong N(K,L)/T(K,L)$, this yields that $n_{e_4}=n_0n_0^Jh$ for some $h\in T(K,L)$, so $n_{e_4}\in n_0L_J(K,L)$.
Since $n_{e_4} = u_{e_4}(1) u_{-e_4}(-1) u_{e_4}(1)$, and each of these three factors is fixed by $\sigma$, we also have $n_{e_4} \in G^1$, proving the lemma.
\end{proof}
Next, we define
\[ N^1 := \langle n_0,\ L_J(K,L)\rangle\cap G^1 . \]
This group is certainly non-trivial since $n_{e_4}\in N^1$. 

Similarly as with ordinary Chevalley groups, one can associate to the root sytem $\Phi$ and corresponding vector space $V$ a new root system $\tilde{\Phi}$ and vector space $\tilde{V}$ using the action of $\sigma$ on $V$. Indeed, define $\tilde{V}$ as $C_V(\sigma)\cap J^\perp$, then for every $v\in V$, $\tilde{v}$ denotes the orthogonal projection of $v$ in $\tilde{V}$. One can prove that $\tilde{\Phi}:=\{\tilde{r}\mid r\in\Phi\setminus \Phi_J\}$ forms a (not necessarily reduced) root system of $\tilde{V}$.
We denote the Weyl group corresponding to $\tilde{\Phi}$ by $\tilde{W}$.

\begin{lemma}\label{le:N1}
Let $n^1\in N^1$, then there exists an element $n\in N(K,L)$ such that $n^1L_J(K,L)=nL_J(K,L)$ where $nT(K,L)$ corresponds to the shortest element $w$ in $wW_J$. Then $w\in C_W(\sigma)\cap\Stab(\Phi_J)$ and $w|_{\tilde{V}}\in \tilde{W}$. The map 
\[\phi:N^1/H^1\to\tilde{W};n^1H^1\mapsto w|_{\tilde{V}}\]
is an isomorphism.
\end{lemma}
\begin{proof}
	Again, the proof of \cite[Lemma 2.9]{Steinbach1} holds almost verbatim.
\end{proof}

In the current setting, we simply have $\tilde V = C_V(\sigma) = J^\perp = \R e_4$.
The set $\Phi\setminus\Phi_J$ consists precisely of those roots that have a non-zero coefficient for $e_4$,
and we find that $\tilde{\Phi}$ is a root system of type $\mathsf{BC}_1$; in particular, the corresponding Weyl group $\tilde{W}$ has order $2$.
From this, we can immediately conclude that $N^1=\langle n_{e_4}\rangle H^1$.

Before we can proceed with the actual proof of the existence of a split saturated BN-pair, we formulate a last important theorem, again inspired by \cite{Steinbach1}.
\begin{theorem}\label{uniquedecomposition}
Every $g\in G^1\setminus B^1$ can be written as 
\[g=uln_{e_4}u'\] 
with $u\in U_J(K,L)\cap\Fix(\sigma)$, $u'\in {U^{-}_{w_{e_4},J}}\cap\Fix(\sigma)$ and $l\in L_J(K,L)\cap \Fix(\sigma)$.
\end{theorem}
\begin{proof}
	Using Lemma~\ref{Stab} and Lemma~\ref{le:N1}, the proof can be taken over from \cite[Theorem 2.10]{Steinbach1}.
\end{proof}

We now have enough information to prove the existence of a split saturated BN-pair.
\begin{theorem} $G^1$ together with $(B^1,N^1)$ forms a saturated, split BN-pair of rank one.
\end{theorem}
\begin{proof}
We check that all five conditions are satisfied. 
\begin{compactenum}[(i)]
    \item
	We show that $G^1=\langle B^1,N^1\rangle$. This follows immediately from Theorem~\ref{uniquedecomposition}.
    \item
    We first prove that $H^1=B^1 \cap N^1$. It is easy to see that $H^1\subseteq B^1 \cap N^1$. If on the other hand $n_{e_4}h\in B^1$ for some $h\in H^1$, this would imply that $n_{e_4}\in P_J(K,L)$, a contradiction. From this it follows immediately that $H^1\unlhd N^1$ since $|\tilde{W}|=2$ and therefore $[N^1:H^1]=2$.
        \item
        The element $\omega:=n_{e_4}\in N^1\setminus H^1$ with $n_{e_4}^2=e$ such that $N^{1}=\langle H^1,n_{e_4}\rangle$. We also have $G^1=B^1\cup B^1n_{e_4} B^1$ since $N^1=n_{e_4}H^{1}\cup H^1$ and $n_{e_4} B^1n_{e_4}\neq B^1$ because $x_{-e_4}(1)\notin B^1\subseteq P_J$ .
        \item The group $U^1\unlhd B^1$ since $U_J(K,L)\cap\Fix(\sigma)\unlhd P_J(K,L)\cap G^1$. As $B^1\subseteq P_J(K,L)$ with $P_J(K,L)=U_J(K,L)\rtimes L_J(K,L)$, we find that every $b\in B^1$ can be written uniquely as $b=ul$ for some $u\in U_J(K,L)$ and $l\in L_J(K,L)$. Now $\sigma$ fixes $U_J(K,L)$ and $L_J(K,L)$ which means that $u,l$ have to fixed by $\sigma$ as well, so we get that $u\in U^1$ and $l\in H^1$. This shows that $B^1=U^1\rtimes H^1$.
                \item Because $H^1\unlhd N^1$, we obtain that $n_{e_4}H^1n_{e_4}=H^1$, so $H^1\subseteq B^1\cap n_{e_4}B^1n_{e_4}$.                
                Remains to check the direction $B^1\cap (n_{e_4}B^{1}n_{e_4})\subseteq H^1$. We have that $x\in B^1\cap (n_{e_4}B^{1}n_{e_4})\subseteq B^1$ so the only thing left to check is that $x$ belongs to $N^1$. Since $x$ is in the intersection of the above groups, we can write $x=b_1=n_{e_4}b_2n_{e_4}=n_{e_4}uhn_{e_4}=n_{e_4}un_{e_4}h'$ for certain $b_1,b_2 \in B^1$, $u\in U^1$ and $h,h'\in H^1$. This implies $b_1h'^{-1}\in V^1\cap B^1=\{e\}$ or $b_1=h'\in H^1$.\qedhere
\end{compactenum}\end{proof}

\section{The Moufang set of mixed type $\F_4$}\label{se:mixedF4}

From Lemma \ref{BN-pair} we find that the set $X:=\{(U^1)^g\mid g\in G^1\}$ together with the set of subgroups $\{(U^1)^g\mid g\in G^1\}$ acting on $X$ by conjugation, forms a Moufang set $\mouf=(X, (V_x)_{x\in X})$.

In general, every Moufang set of the form $\mouf=(X,(U_x)_{x\in X})$ can be written as $\mouf(U,\tau)$ for some group $U$ and some permutation $\tau$ on $X$ using Lemma \ref{Utau}.
We use this lemma to find a representation of our Moufang set in terms of a group $(U,+)$ and a permutation $\tau$. 

We choose $\infty:=U^1$ and $0:={(U^1)}^{n_{e_4}}=V^1$, and we define
\[ U := \{{(U^1)}^g\mid g\in G^1\}\setminus \{U^1\} . \]
Notice that every $g\in G^1\setminus B^1$ can be written in a unique way as $g=bn_{e_4}u$ with $b\in B^1$ and $u\in U^1$;
equivalently, for any two elements $g=bn_{e_4}u$ and $g'=b'n_{e_4}u'$ with $b,b'\in B^1,\ u,u'\in U^1$,
we have ${(U^1)}^g={(U^1)}^{g'}$ if and only if $u=u'$.
Therefore, the map
\[
\zeta \colon U^1 \to U \colon u \mapsto (U^1)^{n_{e_4}u}
\]
is a bijection.
In particular, this makes $U$ into a group which is isomorphic to $U^1$.
Finally, we can set $\tau:=n_{e_4}$, which acts on the set $U^*$ of non-trivial elements of $U$ by conjugation.

We conclude that the corresponding Moufang set $\mouf$ is given by $\mouf = \mouf(U,n_{e_4})$.
In section~\ref{ss:U1}, we will determine the group $U$;
in section~\ref{ss:tau}, we will determine the action of $\tau$ on $U^*$.

\subsection{Description of the group $U$}\label{ss:U1}

We determine what an arbitrary element of $U^1$ looks like, 
and we will describe the group structure of $U \cong U^1$.

Since $U^1=U_J\cap \Fix(\sigma)$, we find
\begin{align*}
 U^1:=U_{r_1}U_{r_2}\dotsm U_{r_{15}} \cap \Fix_{\F_4(K,L)}{(\sigma)},
\end{align*} 
with
\begin{align*}
\begin{array}{lll}
r_1=e_4&\\
r_2=e_1+e_4 && r_3=-e_1+e_4\\
r_4=e_2+e_4 && r_5=-e_2+e_4\\
r_6=e_3+e_4 && r_7=-e_3+e_4\\
r_8=\frac{1}{2}(e_1+e_2-e_3+e_4) && r_9=\frac{1}{2}(-e_1-e_2+e_3+e_4)\\
r_{10}=\frac{1}{2}(e_1-e_2+e_3+e_4) && r_{11}=\frac{1}{2}(-e_1+e_2-e_3+e_4)\\
r_{12}=\frac{1}{2}(-e_1+e_2+e_3+e_4) && r_{13}=\frac{1}{2}(e_1-e_2+e_3+e_4)\\
r_{14}=\frac{1}{2}(-e_1-e_2-e_3+e_4) && r_{15}=\frac{1}{2}(e_1+e_2+e_3+e_4),
\end{array}
\end{align*}
$U_r=\{ u_r(t)\mid t\in L\}$ if $r\in\Phi_s$ and $U_r=\{ u_r(t)\mid t\in K\}$ if $r\in\Phi_l$.\\
This implies that an arbitrary element $x$ of $U^{1}$ is of the form
\begin{align*}
u_{r_1}(t_1)u_{r_2}(t_2')u_{r_3}{(t_3')}\dotsm u_{r_7}({t_7'})u_{r_8}(t_8)u_{r_9}(t_9)\dotsm u_{r_{15}}(t_{15})
\end{align*}
with $t_i\in L$ and $t_j'\in K$ for all $i,j$ and satisfies the relation $x=\sigma(x)$.

After rearranging some factors, using the commutator relations in \eqref{eq:comm2}, we find that
\begin{align*}
\sigma(x)=u_{r_1}(c_1\overline{t_1}+\overline{t_8t_9}+\overline{t_{10}t_{11}}+\overline{t_{12}t_{13}}+\overline{t_{14}t_{15}})u_{r_2}(c_3\overline{t'_3})\dotsm u_{r_{15}}(c_{14}\overline{t_{14}})
\end{align*}

We now determine the values for each $c_{r_i}$, using the already known values for $c_{\alpha_1},\dots,c_{\alpha_4}$ from paragraph~\ref{ss:cr} and the product formula. We get
\begin{align*}
\begin{array}{lll}
c_{r_1}=1&&\\
c_{r_2}=\alpha && c_{r_3}=\alpha^{-1}\\
c_{r_4}=\beta && c_{r_5}=\beta^{-1}\\
c_{r_6}=\alpha\beta && c_{r_7}=\alpha^{-1}\beta^{-1}\\
c_{r_8}=1 && c_{r_9}=1\\
c_{r_{10}}=\alpha && c_{r_{11}}=\alpha^{-1}\\
c_{r_{12}}= \beta && c_{r_{13}}=\beta^{-1}\\
c_{r_{14}}=\alpha^{-1}\beta^{-1} && c_{r_{15}}=\alpha\beta.
\end{array}
\end{align*}
and so the relation $x=\sigma(x)$ implies the following relations:
\begin{gather*}
t_1+\overline{t_1}+\overline{t_8}\cdot\overline{t_9}+\overline{t_{10}}\cdot\overline{t_{11}}+\overline{t_{12}}\cdot\overline{t_{13}}+\overline{t_{14}}\cdot\overline{t_{15}}=0,\\
t_3'=\alpha\overline{t_2'},\  t_5'=\beta\overline{t_4'},\ t_7'=\alpha\beta\overline{t_6'} , \\
t_9=\overline{t_8},\ t_{11}=\alpha\overline{t_{10}}, \ t_{13}=\beta\overline{t_{12}},\ t_{14}=\alpha\beta\overline{t_{15}}.
\end{gather*}
Replacing $t'_3,\, t'_5,\, t'_7,\, t_9,\, t_{11},\, t_{13}$ and $t_{14}$ shows that an arbitrary element $x$ of $U^1$ is therefore of the form
\begin{align*}
u_{r_1}(t_1)u_{r_2}(t'_2)u_{r_3}(\alpha \overline{t'_2})\dotsm u_{r_{14}}(\alpha\beta \overline{t_{15}})u_{r_{15}}(t_{15})
\end{align*}
with $t_1+\overline{t_1}+t_8\overline{t_8}+\alpha t_{10}\overline{t_{10}}+\beta t_{12}\overline{t_{12}}+\alpha\beta t_{15}\overline{t_{15}}=0$.

Now let $x,y \in U^1$ be arbitrary, and write
\begin{align*}
	x &= u_{r_1}(t_1)u_{r_2}(t_2')u_{r_3}(t_3')\dotsm u_{r_7}({t_7'})u_{r_8}(t_8)u_{r_9}(t_9)\dotsm u_{r_{15}}(t_{15}), \\
	y &= u_{r_1}(s_1)u_{r_2}(s_2')u_{r_3}(s_3')\dotsm u_{r_7}({s_7'})u_{r_8}(s_8)u_{r_9}(s_9)\dotsm u_{r_{15}}(s_{15});
\end{align*}
then the product $x\cdot y$ is equal to
\begin{multline*}
	u_{r_1}(t_1+s_1+\overline{t_8}s_8+\alpha\overline{t_{10}}s_{10}+\beta\overline{t_{12}}s_{12}+\alpha\beta t_{15}\overline{s_{15}}) \\
	\cdot u_{r_2}(t_2'+s_2')\cdot u_{r_3}(t'_3+s'_3)\dotsm u_{r_{15}}({t_{15}+s_{15}}).
\end{multline*}
To simplify the notation, we will
identify an arbitrary element
\[ x=u_{r_1}(t_1)u_{r_2}(t_2')u_{r_3}(t_3')\dotsm u_{r_7}({t_7'})u_{r_8}(t_8)u_{r_9}(t_9)\dotsm u_{r_{15}}(t_{15}) \]
with the element
\[ ((t_8,t_{10},t_{12},\overline{t_{15}}),(t_1,t'_2,t_4',t'_6)) \in (L\oplus L\oplus L\oplus L)\oplus (L\oplus K\oplus K\oplus K). \]
We conclude that $U^1$ is isomorphic to the group $(U,+)$ with
\begin{multline}\label{eq:U+}
U := \bigl\{ ((x_1,x_2,x_3,x_4),(y_1,y_2,y_3,y_4))\in (L^4) \oplus (L \oplus K^3) \mid \\ 
x_1\overline{x_1}+\alpha x_2\overline{x_2}+\beta x_3\overline{x_3}+\alpha\beta x_4\overline{x_4}+\tr(y_1)=0 \bigr\}, 
\end{multline}
where the group addition $+$ is given by the formula
\begin{multline*}
\bigl( (x_1,x_2,x_3,x_4),(y_1,y_2,y_3,y_4) \bigr) + \bigl( (a_1,a_2,a_3,a_4),(b_1,b_2,b_3,b_4) \bigr) \\
	\shoveleft = \bigl( (x_1+a_1,x_2+a_2,x_3+a_3,x_4+a_4), \\
	(y_1+b_1+\overline{x_1}a_1+\alpha\overline{x_2}a_2+\beta \overline{x_3}a_3+\alpha\beta \overline{x_4}a_4, y_2+b_2, y_3+b_3,\ y_4+b_4) \bigr).
\end{multline*}

\subsection{Description of the action of $\tau$ on $U^*$}\label{ss:tau}

We present a way to calculate the image of $\tau=n_{e_4}$ on an arbitrary non-trivial element $u\in U^1$.
Using the isomorphism with $U$, we need to determine which element of $U^1$ corresponds to the element ${(U^1)}^{n_{e_4}un_{e_4}}$ of $U$.
As mentioned in the previous section, this comes down to rewriting an arbitrary element of the form $g=n_{e_4}un_{e_4}$ with $u\in U^1$ as $g=bn_{e_4}u'$ with $b\in B^1$ and $u'\in U^1$. We have shown in the previous section that this can be done in a unique way and we therefore have $\tau(u)=u'$. 

Bringing such an arbitrary element $n_{e_4}un_{e_4}$ into the right form comes down to quite long, but systematic calculations,
as we will explain below.
We implemented our algorithm in the computer algebra software package {\tt Sage} \cite{sage}.
We refer to \cite{sagecode} for the detailed implementation and the output of the program.

We briefly describe the methods we use to rewrite $n_{e_4}un_{e_4}$ as $bn_{e_4}u'$ for some $b\in B^1$ and $u'\in U^1$.
Suppose $u=x_{r_1}(t_1)x_{r_2}(t_2)\dotsm x_{r_{14}}(t_{14})x_{r_{15}}(t_{15})$, then using the relations
\[n_sx_r(t)n_s=x_{w_s(r)}(t) \quad \text{and} \quad n_{r}(t)=h_r(t)n_r \]
for all $r,s\in \Phi$ and all $t\in L$,
we find that
\begin{multline*} n_{e_4}x_{r_1}(t_1)x_{r_2}(t_2)x_{r_3}(t_3)\dotsm x_{r_{14}}({t_{14}})x_{r_{15}}(t_{15})n_{e_4}\\
=x_{-r_1}(t_1)x_{-r_3}(t_2)x_{-r_2}(t_3)\dotsm x_{-r_{15}}(t_{14})x_{-r_{14}}(t_{15})\end{multline*}
or that this element has the same action on $U^1$ as the element
\[x=n_{e_4}(t_1)x_{r_1}({t_1}^{-1})x_{-r_3}(t_2)x_{-r_2}(t_3)\dotsm  x_{-r_{15}}(t_{14})x_{-r_{14}}(t_{15}),\]
provided $t_1\neq 0$.

\begin{remark}
If $t_1=0$, then the norm condition implies that $t_8,\dots, t_{15}$ are also equal to zero.
Again, we have to distinguish between $t_2 \neq 0$ and $t_2 = 0$ to proceed,
and a similar distinction has to be made for $t_4$ and $t_6$,
but in each case, the process is very similar (but gets easier as more and more elements become zero), and we omit the details.
\end{remark}

We now describe our algorithm to rewrite $x$ in the required form, i.e.\@ in the form
\[ x = b n_{e_4} u' \]
with $b \in B^1$ and $u' \in U^1$.
Our strategy is to try to swap all ``bad root elements'', i.e. all root elements $x_r(t)$ not belonging to $U^1$, and all
``Hua elements'' $h_r(t)$, in $x$, from the right side of $n_{e_4}$ to the left, in such a manner that the elements that we get at the left of $n_{e_4}$ all belong to $P_J$.
At the end, we will then indeed have rewritten $x$ as $b n_{e_4} u'$.
As $n_{e_4}$ and $u'$ are in $G^1$ at the end of the algorithm, so is $b$, and therefore $b\in B^1$ as required.

\paragraph*{Step 1.}

We always start with the leftmost element on the right side of~$n_{e_4}$.
This element will be of one of the following types:
\begin{compactenum}[(1)]
\item  a Hua element $h_r(t)$,
\item  a root element $x_r(t)$, with $r$ containing no term in $e_4$,
\item  a root element $x_r(t)$, with $r$ containing a negative term in $e_4$,
\item a root element $x_r(t)$, with $r$ containing a positive term in $e_4$.
\end{compactenum}
We point out that elements of the the first three types can be swapped to the left side of $n_{e_4}$ without any problem. 
Only if we encounter an element $x_r(t)=x_{r_i}(t)$ of the fourth type, we have a look at the element next to $x_{r_i}(t)$. 

\paragraph*{Step 2.}

Depending on the form of this second element, we can distinguish a few cases.
\begin{compactenum}[(a)]
    \item
	This second element is a Hua element $h_s(t')$.
	In this case, we can use the conjugation relation
	\[ x_{r}(t)h_s(t')=h_s(t')x_{r}\bigl(t\lambda^{-2\langle r,s\rangle/\langle s,s\rangle}\bigr) \]
	to reverse the order of $x_{r}(t)$ and $h_s(t')$ in the product.
    \item
	This second element is of the form $x_r(t')$.
	In this case we simply combine both elements to the single root element $x_r(t + t')$.
    \item
	This second element is of the form $x_s(t')$ with either
	\begin{compactenum}[(i)]
	\item  $s=r_j$, with $j>i$, or
	\item  $s=r_j$, with $j<i$, or
	\item  $s$ contains no positive term in $e_4$.
	\end{compactenum}
	In case (i), there is nothing to do, and we proceed to the next element in the product,
	i.e.\@ the element $x_{r_j}(t')$ now plays the role of $x_r(t)$, and we apply Step 2 on this element.
	In cases (ii) and (iii), we distinguish between the case $s = -r$ (i.e.\@ $s$ and $r$ are opposite roots) and $s \neq -r$.
	If $s \neq -r$, we use the commutator relations to switch the roots $x_{r}(t)$ and $x_s(t')$,
	and we add the possible new element(s) to the right of $x_{s}(t') x_{r}(t)$.
	If on the other hand $s=-r$, we use the equality
	\[ x_{r}(t)x_{-r}(t') = x_{-r}\Bigl(\frac{t'}{tt'+1}\Bigr)\,h_{r}(tt'+1)\,x_{r}\Bigl(\frac{t}{tt'+1}\Bigr) \]
	(when $tt' \neq -1$) to proceed.
	In both cases, we then return to step 1.
	%
\end{compactenum}
By repeating these steps, we end up with an element in $U^1$ on the right side of $n_{e_4}$, as we wanted. 

Applying the algorithm on an arbitrary element of $U^1$ we get, using our {\tt Sage} program \cite{sagecode}, that
the corresponding map $\tau \colon U^1 \to U^1$ (which maps $u$ to $u'$ as explained in the beginning of this section),
expressed in terms of the isomorphic group $(U,+)$ as in~\eqref{eq:U+} on page~\pageref{eq:U+},
is explicitly given by
\begin{multline*}
\tau \colon U \to U  \colon \\ (a,b)\mapsto \bigl( a \cdot \bigl( b+f(a) \bigr)^{-1},
	\bigl( b+f(a) \bigr)^{-1} + f\bigl( a\cdot (b+f(a))^{-1} \bigr) \bigr) ,
\end{multline*}
where
\begin{multline*}
f \colon L\oplus L\oplus L\oplus L\to L\oplus L\oplus L\oplus L \colon \\
	\shoveleft (a_1,a_2,a_3,a_4) \mapsto
		( a_1\overline{a_1}+\alpha a_2\overline{a_2}+\beta a_3\overline{a_3}+\alpha\beta a_4\overline{a_4}, \\
	a_1a_2 + \beta\overline{a_3}a_4, \; a_1a_3+\alpha\overline{a_2}a_4, \; a_2a_3+\overline{a_1}a_4 ),
\end{multline*}
and where the multiplication of the elements in $L^4$ is the octonion multiplication described in Remark~\ref{re:octmult}.

\subsection{Conclusion}

We now summarize our results.
\begin{theorem}\label{th:main}
	Let $(k,\ell)$ be a pair of fields of characteristic $2$ such that $\ell^2 \leq k \leq \ell$.
	Let $\oct$ be an octonion division algebra over $k$,
	with norm $\N$, and let $\oct_\ell = \oct \otimes_k \ell$, with norm $\N_\ell$.
	Let $K$ and $L$ be separable quadratic field extensions of $k$ and $\ell$, respectively,
	such that $L^2 \leq K \leq L \leq \oct_\ell$,
	and identify $\oct_\ell$ with $L^4$.
	Under this identification, we define a subspace $\oct_\mixed := L \oplus K^3$ of $\oct_\ell$.
	Assume that the restriction of $\N_\ell$ to $\oct_\mixed$ is anisotropic.

	There exist constants $\alpha,\beta \in k^\times$ such that the norm $\N$ is given by
	\[ \N \colon \oct_\ell = L^4 \to K \colon (a_1,a_2,a_3,a_4) \mapsto
		a_1\overline{a_1}+\alpha a_2\overline{a_2}+\beta a_3\overline{a_3}+\alpha\beta a_4\overline{a_4} . \]
	Let
	\begin{multline*}
	f \colon \oct_\ell \to \oct_\ell \colon
		(a_1,a_2,a_3,a_4) \mapsto
			\bigl( \N(a_1,a_2,a_3,a_4), \\
		a_1a_2 + \beta\overline{a_3}a_4, \; a_1a_3+\alpha\overline{a_2}a_4, \; a_2a_3+\overline{a_1}a_4 \bigr) ,
	\end{multline*}
	and let
	\begin{multline*}
	g \colon \oct_\ell \times \oct_\ell \to L \colon
		\bigl( (a_1,a_2,a_3,a_4), (b_1,b_2,b_3,b_4) \bigr) \mapsto \\
		\overline{a_1}b_1 + \alpha\overline{a_2}b_2 + \beta \overline{a_3}b_3 + \alpha\beta \overline{a_4}b_4.
	\end{multline*}
	Define
	\[
	U := \bigl\{ (a,b) \in \oct_\ell \oplus \oct_\mixed \mid \N(a) + \tr(b)=0 \bigr\} ,
	\]
	and make $U$ into a group by setting
	\[ (a,b) + (c,d) := (a+c, b+d+g(a,c)) \]
	for all $(a,b), (c,d) \in U$.
	Define a permutation $\tau$ on $U^*$ by
	\[
	\tau(a,b) := \bigl( a \cdot \bigl( b+f(a) \bigr)^{-1},
		\bigl( b+f(a) \bigr)^{-1} + f\bigl( a\cdot (b+f(a))^{-1} \bigr) \bigr)
	\]
	for all $(a,b) \in U$.
	Then $\mouf(U, \tau)$ is a Moufang set corresponding to a mixed group of type $\F_4$.
\end{theorem}

\section{Algebraic Moufang sets of type $\F_4$ in characteristic $2$}\label{se:alg}

We conclude with having a closer look at what happens to the above results when $K=L$, 
which is in fact the algebraic case in characteristic $2$.
More precisely, we show that every Moufang set $\mouf(U,\tau)$ we obtained in section~\ref{se:mixedF4} in this case is isomorphic to an algebraic Moufang set of type $\F_4$ (as we expect).

To see this, we apply the transformation $\varphi$ on
$U = \bigl\{ (a,b) \in \oct_\ell \oplus \oct_\ell \mid \N(a) + \tr(b)=0 \bigr\}$
with
\begin{align*}
\varphi \colon U\to U \colon (a,b)\mapsto (a,b+f(a))
\end{align*}
Then $\varphi((a,b)+(c,d))=\varphi(a,b)\ \tilde{+}\ \varphi(c,d)$ with
\begin{align*}
(x_1,y_1)\ \tilde{+}\ (x_2,y_2)=(x_1+x_2, y_1+y_2+\overline{x_2}\cdot x_1)
\end{align*}
for all $(x_1,y_1),\ (x_2,y_2)\in U$. Furthermore, $\varphi(\tau((x,y))=\tilde{\tau}(x,y)$ for all $(x,y)\in U$ with
\begin{align*}
\tilde{\tau} \colon U^*\to U^* \colon (x,y)\mapsto (x\cdot y^{-1},y^{-1}).
\end{align*}
We find that $\mouf(U,\tau)$ is isomorphic to $\mouf(U,\tilde{\tau})$, which is indeed a Moufang set of type $\F_4$;
see \cite[Theorem 2.1]{DVM}.

\begin{remark}
It is an interesting open question whether or not all Moufang sets of mixed type $\F_4$ can be embedded in some algebraic Moufang set of type $\F_4$.
This problem can be reformulated to the question whether the extension of the mixed norm form $\N$ on $\oct_\mixed$ to the larger space $\oct_\ell$ remains anisotropic.
\end{remark}


\bigskip\hrule\bigskip
\small

\noindent
Elizabeth Callens \\
Ghent University, Dept.\@ of Mathematics \\
Krijgslaan 281 (S22), B-9000 Gent, Belgium \\
{\tt elcallen@cage.UGent.be}

\medskip

\noindent
Tom De Medts \\
Ghent University, Dept.\@ of Mathematics \\
Krijgslaan 281 (S22), B-9000 Gent, Belgium \\
{\tt tdemedts@cage.UGent.be}

\end{document}